\documentclass[11pt,reqno]{amsart}

\usepackage{color}
\usepackage{amsmath, amsthm, amssymb}
\usepackage{amsfonts}
\usepackage{enumitem}
\usepackage[ansinew]{inputenc}
\usepackage[dvips]{epsfig}
\usepackage{graphicx}
\usepackage[english]{babel}
\usepackage{hyperref}
\theoremstyle{plain}
\newtheorem{thm}{Theorem}[section]
\newtheorem{cor}[thm]{Corollary}
\newtheorem{lem}[thm]{Lemma}
\newtheorem{prop}[thm]{Proposition}

\theoremstyle{definition}
\newtheorem{defi}[thm]{Definition}

\theoremstyle{remark}
\newtheorem{rem}[thm]{Remark}

\numberwithin{equation}{section}

\newcommand{\de}{\partial}

\newcommand{\R}{\mathbb{R}}

\newcommand{\N}{\mathbb{N}}

\newcommand{\eps}{\varepsilon}

\newcommand{\average}{{\mathchoice {\kern1ex\vcenter{\hrule height.4pt
width 6pt depth0pt} \kern-9.7pt} {\kern1ex\vcenter{\hrule
height.4pt width 4.3pt depth0pt} \kern-7pt} {} {} }}

\def\R{\mathbb{R}}

\begin{document}

\title[Free boundary regularity for a.e. solution to the Signorini problem]{Free boundary regularity for almost every solution to the Signorini problem}

\author{Xavier Fern\'andez-Real}
\address{EPFL SB, Station 8, CH-1015 Lausanne, Switzerland}
\email{xavier.fernandez-real@epfl.ch}

\author{Xavier Ros-Oton}
\address{ICREA, Pg. Llu\'is Companys 23, 08010 Barcelona, Spain \& Universitat de Barcelona, Departament de Matem\`atiques i Inform\`atica, Gran Via de les Corts Catalanes 585, 08007 Barcelona, Spain. }
\email{xros@ub.edu}



\thanks{This work has received funding from the European Research Council (ERC) under the Grant Agreements No 721675 and No 801867. 
In addition, X. F. was supported by the SNF grant 200021\_182565 and X.R. was supported by the SNF grant 200021\_178795 and by the MINECO grant MTM2017-84214-C2-1-P}

\maketitle

\begin{abstract}
We investigate the regularity of the free boundary for the Signorini problem in $\R^{n+1}$.
It is known that regular points are $(n-1)$-dimensional and $C^\infty$. However, even for $C^\infty$ obstacles $\varphi$, the set of non-regular (or degenerate) points could be very large --- e.g. with infinite $\mathcal{H}^{n-1}$ measure.

The only two assumptions under which a nice structure result for degenerate points has been established are: when $\varphi$ is analytic, and when $\Delta\varphi < 0$. However, even in these cases, the set of degenerate points is in general $(n-1)$-dimensional --- as large as the set of regular points.

In this work, we show for the first time that, ``usually'', the set of degenerate points is \emph{small}. Namely, we prove that, given any $C^\infty$ obstacle, for \emph{almost every} solution the non-regular part of the free boundary is at most $(n-2)$-dimensional. This is the first result in this direction for the Signorini problem.

Furthermore, we prove analogous results for the obstacle problem for the fractional Laplacian $(-\Delta)^s$, and for the parabolic Signorini problem. In the parabolic Signorini problem, our main result establishes that the non-regular part of the free boundary is $(n-1-\alpha_\circ)$-dimensional for almost all times $t$, for some $\alpha_\circ > 0$.

Finally, we construct some new examples of free boundaries with degenerate points.
\end{abstract}

\vspace{5mm}

\section{Introduction}

The Signorini problem (also known as the thin or boundary obstacle problem) is a classical free boundary problem that was originally studied by Antonio Signorini in connection with linear elasticity \cite{Sig33, Sig59, KO88}.
The problem gained further attention in the seventies due to its connection to mechanics, biology, and even finance --- see \cite{DL76}, \cite{Mer76,CT04}, and \cite{Fer21, Ros18} ---, and since then it has been widely studied in the mathematical community; see \cite{Caf79, AC04, CS07, ACS08, GP09, PSU12, KPS15, KRS19, DGPT17, FS18, CSV19, JN17, FJ18, Shi18} and references therein.

\vspace{2mm}

The main goal of this work is to better understand the size and structure of the non-regular part of the free boundary for such problem.

In particular, our goal is to prove for the first time that, for \emph{almost every} solution (see Remark~\ref{rem.prevalence}), the set of non-regular points is \emph{small}. 
As explained in detail below, this is completely new even when the obstacle $\varphi$ is analytic or when it satisfies $\Delta \varphi < 0$.

\vspace{1mm}

\subsection{The Signorini problem}

Let us denote $x = (x', x_{n+1}) \in \R^n\times\R$ and $B_1^+ = B_1\cap \{x_{n+1} > 0\}$. We say that $u \in H^1(B_1^+)$ is a solution to the Signorini problem with a smooth obstacle $\varphi $ defined on $B_1' := B_1\cap \{x_{n+1} = 0\}$ if $u$ solves
\begin{equation}
\label{eq.thinobst_intro}
  \left\{ \begin{array}{rcll}
  \Delta u&=&0 & \textrm{ in } B_1^+\\
  \min\{-\de_{x_{n+1}} u, u-\varphi\} & = &0 & \textrm{ on } B_1\cap \{x_{{n+1}} = 0\},
  \end{array}\right.
\end{equation}
in the weak sense, for some boundary data $g\in C^0(\de B_1\cap \{x_{n+1} \ge 0\})$. Solutions to the Signorini problem are minimizers of the Dirichlet energy
\[
\int_{B_1^+} |\nabla u |^2,
\]
under the constrain $u \ge \varphi$ on $\{x_{n+1} = 0\}$, and with boundary conditions $u = g$ on $\de B_1 \cap \{x_{n+1} > 0\}$.

Problem \eqref{eq.thinobst_intro} is a {\em free boundary problem}, i.e., the unknowns of the problem are the solution itself, and the contact set
\[
\Lambda(u) := \big\{x'\in \R^n: u(x', 0) = \varphi(x')\big\}\times\{0\}\subset \R^{n+1},
\]
 whose topological boundary in the relative topology of $\R^n$, which we denote $\Gamma(u) = \de\Lambda(u) = \de \{x'\in \R^n: u(x', 0) = \varphi(x')\}\times\{0\}$, is known as the \emph{free boundary}.

Solutions to \eqref{eq.thinobst_intro} are known to be $C^{1,\frac12}$ (see \cite{AC04}), and this is optimal.

\subsection{The free boundary}

While the optimal regularity of the solution is already known, the structure and regularity of the free boundary is still not completely understood. The main known results are the following.

The free boundary can be divided into two sets,
\[
\Gamma(u) = {\rm Reg}(u) \cup {\rm Deg}(u),
\]
the set of {\em regular points},
\[
{\rm Reg}(u) := \left\{x = (x', 0)\in \Gamma(u) : 0<c r^{3/2}\le \sup_{B_r'(x')} (u - \varphi) \le C r^{3/2},\quad\forall r\in (0, r_\circ)\right\},
\]
and the set of non-regular points or {\em degenerate points}
\begin{equation}
\label{eq.degu}
{\rm Deg}(u) := \left\{x = (x', 0)\in \Gamma(u) : 0\le \sup_{B_r'(x')} (u - \varphi) \le Cr^{2},\quad\forall r\in (0, r_\circ)\right\},
\end{equation}
(see \cite{ACS08}). Alternatively, each of the subsets can be defined according to the order of the blow-up at that point. Namely, the set of regular points are those whose blow-up is of order $\frac32$, and the set of degenerate points are those whose blow-up is of order $\kappa$ for some $\kappa \in [2, \infty]$.

Let us denote $\Gamma_\kappa$ the set of free boundary points of order $\kappa$. That is, those points whose blow-up is homogeneous of order $\kappa$ (we will be more precise about it later on, in Section~\ref{sec.Sch}; the definition of $\Gamma_\infty$ is slightly different). Then, it is well known that the free boundary can be divided as
\begin{equation}
\label{eq.FBst}
\Gamma(u) = \Gamma_{3/2} \cup \Gamma_{\rm even} \cup \Gamma_{\rm odd} \cup \Gamma_{\rm half} \cup \Gamma_* \cup \Gamma_{\infty},
\end{equation}
where:
\begin{itemize}[leftmargin=5.5mm]
\item $ \Gamma_{3/2} = {\rm Reg}(u)$ is the set of regular points. They are an open $(n-1)$-dimensional subset of $\Gamma(u)$, and it is $C^\infty$ (see \cite{ACS08, KPS15, DS16}).

\item $\Gamma_{\rm even} = \bigcup_{m\ge 1} \Gamma_{2m}(u)$ denotes the set of points whose blow-ups have even homogeneity. Equivalently, they can also be characterised as those points of the free boundary where the contact set has zero density, and they are often called singular points. They are contained in the countable union of $C^1$ $(n-1)$-dimensional manifolds; see \cite{GP09, FJ18}.

\item $\Gamma_{\rm odd}= \bigcup_{m\ge 1} \Gamma_{2m+1}(u)$ is, a priori, also an at most $(n-1)$-dimensional subset of the free boundary and it is $(n-1)$-rectifiable (see \cite{FS18, KW13, FS19,FRS19}), although it is not actually known whether it exists.

\item  $\Gamma_{\rm half} = \bigcup_{m\ge 1} \Gamma_{2m+3/2}(u)$ corresponds to those points with blow-up of order $\frac72$, $\frac{11}{2}$, etc. They are much less understood than regular points. The set $\Gamma_{\rm half}$ is an $(n-1)$-dimensional subset of the free boundary and it is $(n-1)$-rectifiable (see \cite{FS18, KW13, FS19}).

\item $\Gamma_*$ is the set of all points with homogeneities $\kappa\in (2, \infty)$, with $\kappa\notin \N$ and $\kappa\notin 2\N-\frac12$. This set has Hausdorff dimension at most $n-2$, so it is always \emph{small}, see \cite{FS18, KW13, FS19}.

\item  $\Gamma_{\infty}$ is the set of points with infinite order (namely, those points at which $u-\varphi$ vanishes at infinite order, see \eqref{eq.Ginf}). For general $C^\infty$ obstacles it could be a huge set, even a fractal set of infinite perimeter with dimension exceeding $n-1$. When $\varphi$ is analytic, instead, $\Gamma_\infty$ is empty.
\end{itemize}

Overall, we see that, for general $C^\infty$ obstacles, the free boundary could be really irregular.

The only two assumptions under which a better regularity is known are:
\begin{itemize}\setlength\itemsep{0.1cm}
\item[$\circ$] $\Delta \varphi < 0$ on $B_1'$ and $u = 0$ on $\de B_1\cap \{x_{n+1} > 0\}$. In this case, $\Gamma(u) = \Gamma_{3/2}\cup\Gamma_2$ and the set of degenerate points is locally contained in a $C^1$ manifold; see \cite{BFR18}.
\item[$\circ$] $\varphi$ is analytic.
In this case, $\Gamma_\infty = \varnothing$ and $\Gamma$ is $(n-1)$-rectifiable, in the sense that it is contained in a countable union of $C^1$ manifolds, up to a set of zero $\mathcal{H}^{n-1}$-measure, see \cite{FS18,KW13}.
\end{itemize}

The goal of this paper is to show that, actually, for {\em most} solutions, {\em all} the sets $\Gamma_{\rm even}$, $\Gamma_{\rm odd}$, $\Gamma_{\rm half}$, and $\Gamma_\infty$ are {\em small}, namely, of dimension at most $n-2$. This is new even in case that $\varphi$ is analytic and $\Delta \varphi < 0$.

\subsection{Our results}
%
%

\label{ssec.ourresults}
We will prove here that, even if degenerate points could potentially constitute a large part of the free boundary (of the same dimension as the regular part, or even higher), they are not {\em common}. More precisely, for almost every obstacle (or for almost every boundary datum), the set of degenerate points is {\em small}. This is the first result in this direction for the Signorini problem, even for zero obstacle.

Let $g_\lambda\in C^0(\de B_1)$ for $\lambda\in [0, 1]$, and let us denote by $u_\lambda$ the family of solutions to \eqref{eq.thinobst_intro}, satisfying
\begin{equation}
\label{eq.bdrydata}
u_\lambda = g_\lambda,\quad\textrm{on}\quad \de B_1\cap \{x_{n+1} > 0\},
\end{equation}
with $g_\lambda$ satisfying
\begin{equation}
\label{eq.cond_glam}
\begin{array}{rcll}
g_{\lambda+\eps}& \ge &g_\lambda,&\textrm{on}\quad \de B_1\cap\{x_{n+1} > 0\}\\
g_{\lambda+\eps} &\ge &g_\lambda + \eps&\textrm{on}\quad \de B_1\cap \{x_{n+1} \ge \frac12\},
\end{array}
\end{equation}
for all $\lambda \in [0, 1)$, $\eps\in (0, 1-\lambda)$.

Our main result reads as follows.

\begin{thm}
\label{thm.MAIN0}
Let $u_\lambda$ be any family of solutions of \eqref{eq.thinobst_intro} satisfying \eqref{eq.bdrydata}-\eqref{eq.cond_glam}, for some obstacle $\varphi\in C^\infty$.
Then, we have
\[
\dim_{\rm \mathcal{H}} \big({\rm Deg}(u_\lambda)\big) \le n-2\quad\textrm{for a.e.}\quad\lambda\in [0, 1],
\]
where ${\rm Deg}(u_\lambda)$ is defined by \eqref{eq.degu}.

In other words, for a.e. $\lambda\in [0, 1]$, the free boundary $\Gamma(u_\lambda)$ is a $C^\infty$ $(n-1)$-dimensional manifold, up to a closed subset of Hausdorff dimension $n-2$.
\end{thm}

This result is completely new even for analytic obstacles, or for $\varphi = 0$. No result of this type was known for the Signorini problem.

The results we prove (see Theorem~\ref{thm.main000} and Proposition~\ref{prop.2mgrow}) are actually more precise and concern the Hausdorff dimension of $\Gamma_{\ge \kappa}(u_\lambda)$, the set of points of order greater or equal than $\kappa$. We will show that, if $3\le \kappa \le n+1$, then $\Gamma_{\ge \kappa}(u_\lambda)$ has dimension $n-\kappa+1$, while for $\kappa > n+1$, then $\Gamma_{\ge \kappa}(u_\lambda)$ is empty for almost every $\lambda \in [0, 1]$. We refer to \cite[Chapter 4]{Mat95} for the definition of Hausdorff dimension. 

Theorem \ref{thm.MAIN0} also holds true for non-smooth obstacles. Namely, we will prove that for $\varphi\in C^{3, 1}$ we have $\dim_{\rm \mathcal{H}} \left({\rm Deg}(u_\lambda)\right) \le n-2$ for a.e. $\lambda\in [0, 1]$.
In particular, the free boundary $\Gamma(u_\lambda)$ is  $C^{2, \alpha}$  up to a subset of dimension $n-2$ for a.e. $\lambda\in [0, 1]$; see \cite{JN17, KPS15, AR19}.

\begin{rem}
\label{rem.prevalence}
In the context of the theory of prevalence, \cite{HSY92} (see also \cite{OY05}), Theorem~\ref{thm.MAIN0} says that the set of solutions satisfying that the free boundary has a small degenerate set is \emph{prevalent} within the set of solutions (say, given by $C^0$ or $L^\infty$ boundary data). Alternatively, the set of solutions whose degenerate set is not lower dimensional is \emph{shy}.

In particular, we can say that for \emph{almost every} boundary data (see \cite[Definition 3.1]{OY05}) the corresponding solution has a lower dimensional degenerate set. This is because adding a constant as in \eqref{eq.cond_glam} is a \emph{1-probe} (see \cite[Definition 3.5]{OY05}) for the set of boundary data, thanks to Theorem~\ref{thm.MAIN0}.
\end{rem}

We will establish the following finer result regarding the set $\Gamma_{\infty}(u_\lambda)$. While it is known that it can certainly exist for some solutions $u_\lambda$ (see Proposition~\ref{prop.inf_points}), we show that it will be empty for almost every $\lambda\in[0, 1]$:

\begin{thm}
\label{thm.MAIN01}
Let $u_\lambda$ be any family of solutions of \eqref{eq.thinobst_intro} satisfying \eqref{eq.bdrydata}-\eqref{eq.cond_glam}, for some obstacle $\varphi\in C^\infty$. Then, there exists $\mathcal{E}\subset [0,1]$ such that $\dim_{\mathcal{H}} \mathcal{E} = 0$ and
\[
\Gamma_{\infty}(u_\lambda) = \varnothing,
\]
for every $\lambda \in [0, 1]\setminus \mathcal{E}$.

Furthermore, for every $h > 0$, there exists some $\mathcal{E}_h\subset [0, 1]$ such that $\dim_{\mathcal{M}} \mathcal{E}_h = 0$ and 
\[
\Gamma_\infty(u_\lambda)\cap B_{1-h} = \varnothing,
\]
for every $\lambda\in [0, 1]\setminus \mathcal{E}_h$. 
\end{thm}

We remark that in the previous result, $\dim_{\mathcal{H}}$ denotes the Hausdorff dimension, whereas $\dim_{\mathcal{M}}$ denotes the Minkowski dimension (we refer to \cite[Chapters 4 and 5]{Mat95}). As such, the second part of the result is much stronger than the first one (e.g., $0 = \dim_{\mathcal{H}}\big(\mathbb{Q}\cap [0, 1]\big)<\dim_{\mathcal{M}}\big(\mathbb{Q}\cap [0, 1]\big) = 1$).

Let us briefly comment on the condition \eqref{eq.cond_glam}. Notice that such condition can be reformulated in many ways.
In the simplest case, one could simply take $g_\lambda = g_0\pm \lambda$.
Alternatively, one could take a family of obstacles $\varphi_\lambda = \varphi_0 \pm \lambda$ (with fixed boundary conditions); this is equivalent to fixing the obstacle $\varphi_0$ and moving the boundary data $g_\lambda = g\mp \lambda$.
Furthermore, one could also consider $g_\lambda = g_0 + \lambda\Psi$ for any $\Psi \ge 0$, $\Psi\not\equiv 0$.
Then, even if the second condition in \eqref{eq.cond_glam} is not directly fulfilled, a simple use of strong maximum principle makes it true in some smaller ball $B_{1-\rho}$, so that $g_{\lambda+\eps} \ge g_\lambda + c(\rho)\eps$ on $\de B_{1-\rho}\cap \{x_{n+1}\ge \frac12 - \rho/2\}$. By rescaling the function and the domain, we can rewrite it as \eqref{eq.cond_glam}.

\vspace{2mm}

Regularity results for almost every solution have been established before in the context of the classical obstacle problem by Monneau in \cite{Mon03}.
In such problem, however, all free boundary points have homogeneity 2, and non-regular points are characterised by the density of the contact set around them: non-regular points are those at which the contact set has density zero.
In the Signorini problem, instead, the structure of non-regular points is quite different, and they are characterised by the growth of $u$ around them (recall \eqref{eq.degu} and the definition of $\Gamma_{\rm even}$, $\Gamma_{\rm odd}$, $\Gamma_{\rm half}$, and $\Gamma_\infty$).
This is why the approach of \cite{Mon03} cannot work in the present context.

More recently, the results of Monneau for the classical obstacle problem have been widely improved by Figalli, the second author, and Serra in \cite{FRS19}.
The results in \cite{FRS19} are based on very fine higher order expansions at singular points, which then lead to a better understanding of solutions around them, combined with new dimension reduction arguments and a cleaning lemma to get improved bounds on higher order expansions. 

Here, due to the different nature of the problem, we do not need any fine expansion at non-regular points nor any dimension reduction. Most of our arguments require only the growth of solutions at different types of degenerate points, combined with appropriate barriers, and Harnack-type inequalities. The starting point of our results is to use a simple (but key) GMT lemma from \cite{FRS19} (see Lemma~\ref{lem.CL} below).

\subsection{Parabolic Signorini problem}\label{ss.parab}
The previous results use rather general techniques that suitably modified can be applied to other situations. We show here that using a similar approach as in the elliptic case, one can deduce results regarding the size of the non-regular part of the free boundary for the parabolic version of the Signorini problem, for almost every time $t$.

We say that a function $u = u(x, t) \in H^{1, 0}(B_1^+\times (-1, 0])$ (see \cite[Chapter 2]{DGPT17}) solves the parabolic Signorini problem with stationary obstacle $\varphi = \varphi(x)$ if $u$ solves
\begin{equation}
\label{eq.parab_intro}
  \left\{ \begin{array}{rcll}
  \de_t u - \Delta u&=&0 & \textrm{ in } B_1^+\times (-1, 0]\\
  \min\{-\de_{x_{n+1}} u, u-\varphi\} & = &0 & \textrm{ on } B_1\cap \{x_{{n+1}} = 0\} \times (-1, 0 ],
  \end{array}\right.
\end{equation}
in the weak sense (cf. \eqref{eq.thinobst_intro}). A thorough study of the parabolic Signorini problem was made by Danielli, Garofalo, Petrosyan, and To, in \cite{DGPT17}.

The parabolic Signorini problem is a free boundary problem, where the free boundary belongs to $B_1'\times(-1, 0]$ and is defined by
\[
\Gamma(u) := \de_{B_1'\times(-1, 0]}\big\{(x', t) \in B_1'\times(-1, 0] : u(x', 0, t) > \varphi(x')\big\},
\]
where $\de_{B_1'\times(-1, 0]}$ denotes the boundary in the relative topology of ${B_1'\times(-1, 0]}$. Analogously to the elliptic Signorini problem, the free boundary can be divided into regular points and degenerate (or non-regular) points:
\[
\Gamma(u) = {\rm Reg}(u) \cup {\rm Deg}(u).
\]

The set of regular points are those where parabolic blow-ups are parabolically $\frac32$-homogeneous. On the other hand, degenerate points are those where parabolic blow-ups of the solution are parabolically $\kappa$-homogeneous, with $\kappa \ge 2$ (alternatively, the solution detaches at most quadratically from the obstacle in parabolic cylinders, $B_r\times(-r^2, 0]$). Further stratifications according to the homogeneity of the parabolic blow-ups can be done in an analogous way to the elliptic problem, see \cite{DGPT17}.

The set of regular points ${\rm Reg}(u)$ is a relatively open subset of $\Gamma(u)$ and the free boundary is smooth ($C^{1,\alpha}$) around them (see \cite[Chapter 11]{DGPT17}). The set of degenerate points, however, could be even larger than the set of regular points.

In this manuscript we show that, under the appropriate conditions, for a.e. time $t\in (-1, 0]$ the set of degenerate points has dimension $(n-1-\alpha_\circ)$ for some $\alpha_\circ > 0$ depending only on $n$. That is, for a.e. time, the free boundary is mostly comprised of regular points, and therefore, it is smooth almost everywhere.

In order to be able to get results of this type we must impose some conditions on the solution. We will assume that
\begin{equation}
\label{eq.condparab}
u_t> 0\quad\textrm{in}\quad B_1^+\cup \left[(B_1'\times (-1, 0])\cap \{u > \varphi\}\right],
\end{equation}
that is, wherever the solution $u$ is not in contact with the obstacle $\varphi$, it is strictly monotone. Alternatively, by the strong maximum principle, the condition can be rewritten as
\[
\begin{array}{ll}
u_t \ge 0,& \textrm{in } \overline{B_1^+}\times(-1, 0],\\
u_t \ge 1,& \textrm{in } \left(B_1^+\cap \{x_{n+1}\ge 1/2\}\right)\times(-1, 0],
\end{array}
\]
up to a constant multiplicative factor.

Condition~\eqref{eq.condparab} is somewhat necessary. If the strict monotonicity was not required, we could be dealing with a {\em bad} solution (with large non-regular set) of the elliptic problem for a set of times of positive measure, and therefore, we could not expect a result like the one we prove. On the other hand, if one allowed changes in the sign of $u_t$ (alternatively, one allowed {\em non-stationary} obstacles), then the result is also not true (see, for instance, the example discussed in \cite[Figure 12.1]{DGPT17}).

Condition~\eqref{eq.condparab} is actually quite natural. One of the main applications of the parabolic Signorini problem is the study of semi-permeable membranes (see \cite[Section 2.2]{DL76}):

We consider a domain ($B_1^+$) and a thin membrane ($B_1'$), which is semi-permeable: that is, a fluid can pass through $B_1'$ into $B_1^+$ freely, but outflow of the fluid is prevented by the membrane. If we suppose that there is a given liquid pressure applied to the membrane $B_1'$ given by $\varphi$, and we denote $u(x, t)$ the inside pressure of the liquid in $B_1^+$, then the parabolic Signorini problem \eqref{eq.parab_intro} describes the evolution of the inside pressure with time. In particular, since liquid can only enter $B_1^+$ (and we assume no liquid can leave from the other parts of the boundary), pressure inside the domain can only become higher, and the solution will be such that $u_t > 0$. The same condition also appears in volume injection through a semi-permeable wall  (\cite[subsections 2.2.3 and 2.2.4]{DL76}).

Our result reads as follows.

\begin{thm}
\label{thm.parab_MAIN}
Let $\varphi\in C^\infty$ and let $u$ be a solution to \eqref{eq.parab_intro} satisfying \eqref{eq.condparab}. Then,
\[
\dim_{\mathcal{H}} \big({\rm Deg}(u)\cap \{t = t_\circ\} \big) \le n-1-\alpha_\circ\quad\textrm{for a.e.}\quad t_\circ\in (-1, 0],
\]
for some $\alpha_\circ >0 $ depending only on $n$.

In particular, for a.e. $t_\circ \in (-1, 0]$ the free boundary $\Gamma(u)\cap \{t = t_\circ\} $ is a $C^{1,\alpha}$ $(n-1)$-dimensional manifold, up to a closed subset of Hausdorff dimension $n-1-\alpha_0$.
\end{thm}

When $\varphi$ is analytic, then the free boundary is actually $C^\infty$ around regular points.
Higher regularity of the free boundary is also expected for smooth obstacles, but so far it is only known when $\varphi$ is analytic; see \cite{BSZ17}.

It is important to remark that the parabolic case presents some extra difficulties with respect to the elliptic one, and in fact we do not know if a result analogous to Theorem~\ref{thm.MAIN01} holds in this context.
This means that points of order $\infty$ could a priori still appear for all times (even though by Theorem~\ref{thm.parab_MAIN} they are lower-dimensional for almost every time).

\subsection{The fractional obstacle problem}
The Signorini problem in $\R^{n+1}$ can be reformulated in terms of a fractional obstacle problem with operator $(-\Delta)^{\frac12}$ in $\R^n$. Conversely, fractional obstacle problems (with the operator $(-\Delta)^s$, $s\in (0, 1)$) can also be reformulated in terms of thin obstacle problems with weights. In this work we will generally deal with the thin obstacle problem with a weight, so that the results from subsection \ref{ssec.ourresults} can also be formulated for the fractional obstacle problem.

Given an obstacle $\varphi\in C^\infty(\R^n)$ such that
\begin{equation}
\label{eq.varphi_bound}
\{\varphi > 0\}\subset\subset \R^n,
\end{equation}
the fractional obstacle problem with obstacle $\varphi$ in $\R^n$ ($n \ge 2$) is
\begin{equation}
\label{eq.fract_obst}
  \left\{ \begin{array}{rcll}
  (-\Delta)^s v&=&0 & \textrm{ in } \R^n \setminus \{v = \varphi\}\\
   (-\Delta)^s v&\ge&0 & \textrm{ in } \R^n\\
  v&\ge& \varphi & \textrm{ in } \R^n \\
   v(x) & \to  & 0 & \textrm{ as } |x|\to \infty.
  \end{array}\right.
\end{equation}

Solutions to the fractional obstacle problem are $C^{1, s}$ (see \cite{CSS08}). We denote $\Lambda(v) = \{v = \varphi\}$ the contact set, and $\Gamma(v) = \de\Lambda(v) $ the free boundary. As in the Signorini problem (which corresponds to $s = \frac12$) the free boundary can be partitioned into regular points
\[
{\rm Reg}(v) := \left\{x'\in \Gamma(v) : 0< cr^{1+s}\le \sup_{B_r'(x')} (v-\varphi) \le Cr^{1+s},\quad\forall r\in (0, r_\circ)\right\},
\]
and non-regular (or degenerate) points,
\begin{equation}
\label{eq.degv}
{\rm Deg}(v) := \left\{x'\in \Gamma(v) : 0\le \sup_{B_r'(x')} (v - \varphi) \le Cr^{2},\quad\forall r\in (0, r_\circ)\right\}.
\end{equation}
More precisely, if we denote by $\Gamma_{\kappa}(v)$ the free boundary points of order $\kappa$, then the free boundary $\Gamma(v)$ can be further stratified analogously to \eqref{eq.FBst} as
\begin{equation}
\label{eq.FBst_s}
 \Gamma(v) = \Gamma_{1+s} \cup \bigg(\bigcup_{m\ge 1} \Gamma_{2m}\bigg)\cup \bigg(\bigcup_{m\ge 1} \Gamma_{2m+2s}\bigg)\cup \bigg(\bigcup_{m\ge 1} \Gamma_{2m+1+s}\bigg)\cup \Gamma_* \cup \Gamma_{\infty}.
\end{equation}

Here, $\Gamma_{1+s} = {\rm Reg}(v) $ is the set of regular points (\cite{CSS08, Sil07}). Again, it is an open subset of the free boundary,  which is smooth. Similarly, $\Gamma_{2m}$  for $m\ge 1$  are often called singular points, and are those where the contact set has zero measure (see \cite{GR19}). Together with the sets $\Gamma_{2m+2s}$ and $\Gamma_{2m+1+s}$ for $m \ge 1$, they are an $(n-1)$-dimensional rectifiable subset of the free boundary, \cite{GR19, FS19}. Finally, $\Gamma_*$ denotes the set containing the remaining homogeneities (except infinite), and has dimension $n-2$; and $\Gamma_{\infty}$ denotes those boundary points where the solution is approaching the obstacle {\em faster} than any power (i.e., at infinite order). As before, the set $\Gamma_{\infty}$ could have dimension even higher than $n-1$.

The type of result we want to prove in this setting regarding regularity for most solutions is concerned with global perturbations of the obstacle (rather than boundary perturbations, as before). That is, we will consider obstacles fulfilling \eqref{eq.varphi_bound}.

We define the set of solutions indexed by $\lambda\in[0, 1]$ to the fractional obstacle problem as
\begin{equation}
\label{eq.fract_obst_lam}
  \left\{ \begin{array}{rcll}
  (-\Delta)^s v_\lambda &=&0 & \textrm{ in } \R^n \setminus \{v_\lambda = \varphi\}\\
   (-\Delta)^s v_\lambda&\ge&0 & \textrm{ in } \R^n\\
  v_\lambda &\ge& \varphi -\lambda & \textrm{ in } \R^n \\
   v_\lambda (x) & \to  & 0 & \textrm{ as } |x|\to \infty.
  \end{array}\right.
\end{equation}

Then, our main result reads as follows.
\begin{thm}
\label{thm.MAIN1}
Let $v_\lambda$ be any family of solutions solving \eqref{eq.fract_obst_lam}, for some obstacle $\varphi\in C^\infty$ fulfilling \eqref{eq.varphi_bound}. Then, we have 
\[
\dim_{\rm \mathcal{H}} \big({\rm Deg}(v_\lambda)\big) \le n-2,\quad\textrm{for a.e.}\quad\lambda\in [0, 1],
\]
where ${\rm Deg}(v_\lambda)$ is defined by \eqref{eq.degv}.

In other words, for a.e. $\lambda\in [0, 1]$, the free boundary $\Gamma(v_\lambda)$ is a $C^\infty$ $(n-1)$-dimensional   manifold, up to a closed subset of Hausdorff dimension $n-2$.
\end{thm}

As before, we actually prove more precise results (see Theorem~\ref{thm.main000} and Proposition~\ref{prop.2mgrow}). We establish an estimate for the Hausdorff dimension of $\Gamma_{\ge \kappa}(v_\lambda)$. We show that, for $2\le \kappa-2s \le n$, then $\dim_{\mathcal{H}}\Gamma_{\ge \kappa}(v_\lambda)\le n-\kappa+2s$, and if $\kappa > n+2s$, then $\Gamma_{\ge \kappa}(v_\lambda)$ is empty for almost every $\lambda \in [0, 1]$.
Similarly, we can also reduce the regularity of the obstacle to $\varphi\in C^{4, \alpha}$ so that, for a.e. $\lambda\in [0, 1]$, $\dim_{\rm \mathcal{H}} \left({\rm Deg}(v_\lambda)\right) \le n-2$ (in particular, the free boundary $\Gamma(v_\lambda)$ is $C^{3, \alpha}$ up to a subset of dimension $n-2$ for a.e. $\lambda\in [0, 1]$; see \cite{JN17, AR19}).

Theorem~\ref{thm.MAIN1} is analogous to Theorem~\ref{thm.MAIN0}. On the other hand, we also have that:

\begin{thm}
\label{thm.MAIN11}
Let $v_\lambda$ be any family of solutions solving \eqref{eq.fract_obst_lam}, for some obstacle  $\varphi\in C^\infty$  fulfilling \eqref{eq.varphi_bound}. Then, there exists $\mathcal{E}\subset [0,1]$ such that $\dim_{\mathcal{H}} \mathcal{E} = 0$ and
\[
\Gamma_{\infty}(v_\lambda) = \varnothing,
\]
for all $\lambda \in [0, 1]\setminus \mathcal{E}$.

Furthermore, for every $h > 0$, there exists some $\mathcal{E}_h\subset [0, 1]$ such that $\dim_{\mathcal{M}} \mathcal{E}_h = 0$ and 
\[
\Gamma_\infty(v_\lambda)\cap B_{1-h} = \varnothing,
\]
for every $\lambda\in [0, 1]\setminus \mathcal{E}_h$. 
\end{thm}

That is, analogously to Theorem~\ref{thm.MAIN01}, we can also control the size of $\lambda$ for which the free boundary points of infinite order exist.

\subsection{Examples of degenerate free boundary points}
\label{ssec.ex}
Let us finally comment on the non-regular part of the free boundary, that is,
\begin{equation}
\label{eq.degfb}
{\rm Deg}(u) = \Gamma_{\rm even} \cup \Gamma_{\rm odd} \cup \Gamma_{\rm half} \cup \Gamma_* \cup \Gamma_{\infty}.
\end{equation}

The main open questions regarding each of the subsets of the degenerate part of the free boundary are:
\\[0.2cm]
{\bf Q1:} Are there non-trivial examples (e.g., the limit of regular points) of singular points in $\Gamma_{\rm even}$?
\\[0.2cm]
{\bf Q2:} Do points in $\Gamma_{\rm odd}$ exist?
\\[0.2cm]
{\bf Q3:} Can one construct arbitrary contact sets with free boundary formed entirely of $\Gamma_{\rm half}$ (alternatively, do they exist apart from the homogeneous solutions)?
\\[0.2cm]
{\bf Q4:} Do points in $\Gamma_*$ exist?
\\[0.2cm]
{\bf Q5:} How big can the set $\Gamma_\infty$ be?
\\[0.2cm]
In this paper, we answer questions Q1, Q3, and Q5. (Questions Q2 and Q4 remain open.) 

Let us start with Q1. The set $\Gamma_{\rm even} = \bigcup_{m\ge 1} \Gamma_{2m}$, often called the set of singular points, is an $(n-1)$-dimensional subset of the free boundary. Examples of free boundary points belonging to $\Gamma_{\rm even}$ are easy to construct as level sets of homogeneous harmonic polynomials, such as $x_1^2-x_{n+1}^2$, in which case we have $\Gamma = \Gamma_{\rm even} = \{x_1 = 0\}$. They are also expected to appear in less trivial situations but, as far as we know, none has been constructed so far that appears as limit of regular points (i.e., on the boundary of the interior of the contact set). Here, we show that:
\begin{prop}
\label{prop.singpoints}
There exists a  boundary data $g$ such that the free boundary of the solution to the Signorini problem \eqref{eq.thinobst_intro} with $\varphi = 0$ has a sequence of regular points (of order $3/2$) converging to a singular point (of order $2$).
\end{prop}

The proof of the previous result is given in Section~\ref{sec.examples}. In contrast to what occurs with the classical obstacle problem, the construction of singular points does not seem to immediately arise from continuous perturbations of the boundary value under symmetry assumptions. Instead, one has to be aware that there could appear other points (different from regular, but not in $\Gamma_{\rm even}$). Thus, our strategy is based on being in a special setting that avoids the appearance of higher order free boundary points.

On the other hand, regarding question Q3, it is known that examples of such points can be constructed through homogeneous solutions,  in which case they can even appear as limit of regular (or lower frequency) points (see \cite[Example 1]{CSV19}). 
Until now, however, it was not clear whether such points could appear in non-trivial (say, non-homogeneous) situations.

We show that, given \emph{any} smooth domain $\Omega\subset \R^n$, one can find a solution to the Signorini problem whose contact set is exactly given by $\Omega$, and whose free boundary is entirely made of points of order $\frac72$ (or $\frac{11}{2}$, etc.). More generally, we show that given $\Omega$, the contact set for the fractional obstacle problem can be made up entirely of points belonging to $\bigcup_{m\ge 1} \Gamma_{2m+1+s}$ (the case $s = \frac12$ corresponding to the Signorini problem).
\begin{prop}
\label{prop.secondexample}
Let $\Omega\subset\R^n$ be any given $C^\infty$ bounded domain, and let $m\in \N$. Then, there exists an obstacle $\varphi\in C^\infty(\R^n)$ with $\varphi\to 0$ at $\infty$, and a global solution to the obstacle problem
\[
  \left\{ \begin{array}{rcll}
  (-\Delta)^{s} u&\ge&0 & \textrm{ in } \R^n\\
    (-\Delta)^{s} u&=&0 & \textrm{ in } \{u > \varphi\}\\
    u & \ge & \varphi & \textrm{ in } \R^n,\\
    u(x) & \to & 0& \textrm{ as }|x|\to \infty,
  \end{array}\right.
\]
such that the contact set is $\Lambda(u) = \{u = \varphi\} = \Omega$, and all the points on the free boundary $\de\Lambda(u)$ have frequency $2m+1+s$.
\end{prop}
The proof of the previous proposition is constructive: we show a way in which such solutions can be constructed, using some results from \cite{Gru15,AR19}.

Finally, we also answer question Q5, that deals with the set $\Gamma_\infty$. 
Not much has been discussed about it in the literature, though its lack of structure was somewhat known by the community. 
For instance, the following result is not difficult to prove:

\begin{prop}
\label{prop.inf_points}
For any $\eps > 0$ there exists a non-trivial solution $u$ and an obstacle $\varphi\in C^\infty(\R^n)$ such that
\[
  \left\{ \begin{array}{rcll}
  (-\Delta)^{s} u&\ge&0 & \textrm{ in } \R^n\\
    (-\Delta)^{s} u&=&0 & \textrm{ in } \{u > \varphi\}\\
    u & \ge & \varphi & \textrm{ in } \R^n,
  \end{array}\right.
\]
and the boundary of the contact set, $\Lambda(u) = \{u = \varphi\}$, fulfils
\[
\dim_{\mathcal{H}} \de\Lambda(u) \ge n-\eps.
\]
\end{prop}

This shows that, in general, there is no hope to get nice structure results for the full free boundary for $C^\infty$ obstacles. However, thanks to Theorem~\ref{thm.MAIN11} above we know that such behaviour is extremely rare. As before, we are answering question Q5 in the generality of the fractional obstacle problem; the Signorini problem corresponds to the case $s = \frac12$.

\subsection{Organization of the paper}
The paper is organised as follows:

In Section~\ref{sec.Sch} we study the behaviour of degenerate points under perturbation. In particular, we show how the free boundary moves around them when perturbing monotonically the solution to the obstacle problem.
We treat separately general degenerate points, and those of order 2. In Section~\ref{sec.dimGam2} we study the dimension of the set $\Gamma_2$ by means of an appropriate application of Whitney's extension theorem. In Section~\ref{sec.mainresults} we prove the main results of this work, Theorems~\ref{thm.MAIN0}, \ref{thm.MAIN01}, \ref{thm.MAIN1}, and~\ref{thm.MAIN11}.
In Section~\ref{sec.examples} we construct the examples of degenerate points introduced in Subsection~\ref{ssec.ex}, proving Propositions~\ref{prop.singpoints}, \ref{prop.secondexample}, and \ref{prop.inf_points}.
Finally, in Section~\ref{sec.parab} we deal with the parabolic Signorini problem and prove Theorem~\ref{thm.parab_MAIN}.

\section{Behaviour of non-regular points under perturbations}
\label{sec.Sch}

Let $B_1 \subset\R^{n+1}$, $B_1' = \{x'\in \R^n : |x'|< 1\} \subset \R^n$ and let
\begin{equation}
\label{eq.obst}
\varphi :B_1'\to \R,\quad\varphi\in C^{\tau, \alpha}(\overline{B_1'}),\quad\tau \in \N_{\ge 2},~\alpha\in(0,1]
\end{equation}
be our obstacle on the thin space. Let us consider the fractional operator
\[
L_a u := {\rm div}(|x_{n+1}|^a \nabla u)= {\rm div}(|x_{n+1}|^{1-2s} \nabla u),\qquad a := 1-2s,
\]
with $a\in(-1,1)$, and $(0, 1)\ni s = \frac{1-a}{2}$. We will interchangeably use both $a$ and $s$ depending on the situation. (In general, we will use $a$ for the weight exponent, and $s$ for all the other situations.)

Let us suppose that we have a family of {\it increasing} even solutions $u_{\lambda}$ for $0\le \lambda \le 1$ to the fractional obstacle problem
\begin{equation}
\label{eq.thinobst_lam_vp}
  \left\{ \begin{array}{rcll}
 L_a u_\lambda&=&0 & \textrm{ in } B_1\setminus \left(\{x_{{n+1}} = 0\}\cap \{u_\lambda = \varphi\}\right)\\
  L_a u_\lambda&\le&0 & \textrm{ in } {B_1}\\
  u_\lambda &\ge& \varphi & \textrm{ on } \{x_{{n+1}} = 0\},
  \end{array}\right.
\end{equation}
for a given obstacle $\varphi$ satisfying \eqref{eq.obst}. In particular, $\{u_\lambda\}_{0\le \lambda \le 1}$ satisfy
\begin{equation}
\label{eq.uassump}
\begin{array}{rcll}
u_{\lambda}(x', x_{n+1}) & = & u_{\lambda}(x', -x_{n+1})&\quad\textrm{in}\quad B_1, \quad\textrm{for}\quad \lambda\ge 0\\
u_{\lambda'}& \ge& u_{\lambda} &\quad\textrm{in}\quad B_1, \quad\textrm{for}\quad \lambda' \ge \lambda\\
u_{\lambda+\eps}& \ge& u_{\lambda}+\eps &\quad\textrm{in}\quad B_1\cap\{|x_{n+1}|\ge \frac12\}, \quad\textrm{for}\quad \lambda, \eps \ge 0\\
\|u_\lambda\|_{C^{2s}(B_1)}& \le & M, &\quad\textrm{in}\quad B_1 \quad\textrm{for}\quad \lambda\ge 0,
\end{array}
\end{equation}
for some constant $M$ independent of $\lambda$, that will depend on the obstacle (see \eqref{eq.Mf}-\eqref{eq.Mf2} below). Notice that solutions are $C^{1,s}$ in $B_{1/2}'$ (or in $\overline{B_{1/2}^+}$), but only $C^{2s}$ in $B_1$ ($C^{0,1}$ when $s=\frac12$).

We denote $\Lambda(u_\lambda) := \{x' : u_\lambda(x', 0) = \varphi(x')\}\times\{0\}\subset \R^n$ the contact set, and its boundary in the relative topology of $\R^n$, $\de \Lambda(u_\lambda) = \de \{x' : u_\lambda(x', 0) = \varphi(x')\}\times\{0\}$ is the free boundary. Note that, from the monotonicity assumption,
\begin{equation}
\label{eq.subset}
\Lambda(u_\lambda) \subset \Lambda(u_{\lambda'})\quad \textrm{for}\quad\lambda \ge \lambda'.
\end{equation}


\begin{lem}
\label{lem.ulambda}
Let $u_\lambda$ denote the family of solutions to \eqref{eq.thinobst_lam_vp}-\eqref{eq.uassump}. Then, for any $h > 0$ small, $x_\circ\in {B_{1-h}}$, and $\eps > 0$,
\[
\frac{u_{\lambda+\eps}(x_\circ) - u_\lambda(x_\circ)}{\eps} \ge c~{\rm dist}^{2s}(x_\circ, \Lambda(u_\lambda)),
\]
for some constant $c >0$ depending only on $n$, $s$, and $h$. In particular,
\[
\de_\lambda^+ u_\lambda(x_\circ) := \liminf_{\eps\downarrow 0} \frac{u_{\lambda+\eps}(x_\circ) - u_\lambda(x_\circ)}{\eps} \ge c~{\rm dist}^{2s}(x_\circ, \Lambda(u_\lambda)),
\]
for some constant $c>0$ depending only on $n$, $s$, and $h$.
\end{lem}
\begin{proof}
Fix some $\lambda > 0$ and $\eps > 0$, and define
\[
\delta_{\lambda, \eps} u_\lambda(x) = \frac{u_{\lambda +\eps}(x) - u_{\lambda}(x)}{\eps}.
\]
We will show that the result holds for $\delta_{\lambda, \eps} u_\lambda$ for some constant $c$ independent of $\eps> 0$, and in particular, it also holds after taking the $\liminf$.


Notice that $\delta_{\lambda, \eps}u_\lambda (x)\ge 0$ from the monotonicity of $u_\lambda$ in $\lambda$. Notice, also, that $\delta_{\lambda,\eps} u_\lambda \ge 1 $ in $B_1\cap \{x_{n+1} \ge \frac12\}$, form the third condition in \eqref{eq.uassump}. On the other hand,
\[
L_a\delta_{\lambda, \eps} u_\lambda = 0\quad\textrm{in}\quad {B_1}\setminus \Lambda(u_\lambda),
\]
thanks to \eqref{eq.subset}.
%
Now, let
\[
r:= \frac{h}{4} {\rm dist}(x_\circ, \Lambda(u_\lambda))
\]
and we define the barrier function $\psi: {B_1}\to \R$ as the solution to
\[
\left\{
\begin{array}{rcll}
L_a \psi &=& 0& \quad\textrm{in}\quad{B_1}\setminus\{x_{n+1} = 0\}\\
\psi &=& 0& \quad\textrm{on}\quad \{x_{n+1} = 0\}\\
\psi & = & 1&  \quad\textrm{on}\quad\de{B_1}\cap \{|x_{n+1}| \ge \frac12\}\\
\psi & = & 0&  \quad\textrm{on}\quad\de{B_1}\cap \{|x_{n+1}| < \frac12\}.
\end{array}
\right.
\]

Then, by maximum principle,
\[
\delta_{\lambda, \eps} u_\lambda \ge \psi\quad\textrm{in}\quad{B_1}.
\]
Notice that, by the boundary Harnack inequality for Muckenhoupt weights $A_2$ (see  \cite{FJK83}), $\psi$ is comparable to $|x_{n+1}|^{2s}$ (since both vanish continuously at $x_{n+1} = 0$, and both are $a$-harmonic), and in particular, there exists some $c'>0$ small depending only on $n$, $s$, and $h$, such that $\psi\ge c'|x_{n+1}|^{2s}$ in $B_r(x_\circ)$. We have that
\[
L_a \delta_{\lambda, \eps} u_\lambda  = 0,\quad\delta_{\lambda, \eps} u_\lambda \ge \psi\ge c'|x_{n+1}|^{2s}\quad\textrm{in}\quad B_r(x_\circ).
\]
Now, if $x_\circ = (x_\circ', x_{\circ, n+1})$ is such that $|x_{\circ, n+1}|\ge \frac{r}{4}$, it is clear that $\delta_{\lambda, \eps} u_\lambda (x_\circ) \ge c r^{2s}$. On the other hand, if $|x_{\circ, n+1}|\le \frac{r}{4}$, then $L_a \delta_{\lambda, \eps} u_\lambda  = 0$ in $B_{r/2}((x_\circ', 0))$, so that applying Harnack's inequality in $B_{r/4}((x_\circ', 0))$ to $\delta_{\lambda, \eps} u_\lambda $,
\[
\delta_{\lambda, \eps} u_\lambda (x_\circ) \ge \inf_{B_{r/4}((x_\circ', 0))} \delta_{\lambda, \eps} u_\lambda  \ge \frac{1}{C}\sup_{B_{r/4}((x_\circ', 0))} \delta_{\lambda, \eps} u_\lambda  \ge  \frac{c'r^{2s}}{4^{2s}C} = cr^{2s},
\]
for some $c$ depending only on $n$, $s$, and $h$. Thus,
\[
\delta_{\lambda, \eps} u_\lambda (x_\circ) \ge  c r^{2s} = c~{\rm dist}^{2s}(x_\circ, \Lambda(u_\lambda)),
\]
as we wanted to see.
\end{proof}

Let $0\in \de\Lambda(u_\lambda)$ be a  free boundary point for $u_\lambda$. Let us denote $Q_\tau(x')$ the Taylor expansion of $\varphi(x')$ around $0$ up to order $\tau$, and we denote $Q_\tau^a(x)$ its unique even $a$-harmonic extension (see \cite[Lemma 5.2]{GR19}) to $\R^{n+1}$ ($L_a Q_\tau^a(x)= 0$, and $Q_\tau^a(x', 0) = Q_\tau(x')$). Let us define
\[
\bar u_\lambda(x', x_{n+1}) = u_\lambda(x', x_{n+1}) - Q_\tau^a(x', x_{n+1}) + Q_\tau(x') - \varphi(x').
\]
Then $\bar u_\lambda(x', x_{n+1})$ solves the zero obstacle problem with a right-hand side
\begin{equation}
\label{eq.thinobst_lam_vp_f}
  \left\{ \begin{array}{rcll}
  L_a\bar u_\lambda&=&|x_{n+1}|^a f & \textrm{ in } {B_1}\setminus \left(\{x_{{n+1}} = 0\}\cap \{\bar u_\lambda = 0\}\right)\\
  L_a \bar u_\lambda&\le&|x_{n+1}|^a f & \textrm{ in } {B_1}\\
  \bar u_\lambda &\ge& 0 & \textrm{ on } \{x_{{n+1}} = 0\},
  \end{array}\right.
\end{equation}
where
\begin{equation}
\label{eq.Mf}
f = f(x') = \Delta_{x'} (Q_\tau(x') - \varphi(x')).
\end{equation}
In particular, notice that since $Q_\tau(x')$ is the Taylor approximation of $\varphi$ up to order $\tau$, we have that
\begin{equation}
\label{eq.Mf2}
|f(x')|\le M |x'|^{\tau +\alpha -2}
\end{equation}
for some $M > 0$ depending only on $\varphi$. We take $M$ larger if necessary, so that it coincides with the one of \eqref{eq.uassump}.

We consider the generalized frequency formula, for $\theta\in (0, \alpha)$, and for some $C_\theta$ (that is independent of the point around which is taken)
\begin{equation}
\label{eq.genPhi}
\Phi_{\tau, \alpha, \theta}(r, \bar u_\lambda) := (r+C_\theta r^{1+\theta})\frac{d}{dr} \log \max\bigg\{H(r), r^{n+a+2(\tau+\alpha-\theta)}\bigg\},
\end{equation}
where
\[
H(r) := \int_{\de B_r} \bar u_\lambda^2|x_{n+1}|^a.
\]
Then, by \cite[Proposition 6.1]{GR19} (see also \cite{CSS08, GP09}) we know that $\Phi_{\tau, \alpha, \theta}(r, \bar u_\lambda)$ is nondecreasing for $0 < r< r_\circ$ for some $r_\circ$. In particular, $\Phi_{\tau, \alpha, \theta}(0^+, \bar u_\lambda)$ is well defined, and by \cite[Lemma 2.3.2]{GP09},
\[
n+3\le \Phi_{\tau, \alpha, \theta}(0^+, \bar u_\lambda) \le n+a+2(\tau+\alpha-\theta).
\]

We say that $0\in \de\Lambda(u_\lambda)$ is a point of order $\kappa$ if $\Phi_{\tau, \alpha, \theta}(0^+, \bar u_\lambda) = n+1-2s+2\kappa$. In particular, by the previous inequalities
\[
1+s \le \kappa \le \tau+\alpha-\theta
\]
Thanks to \cite[Lemma 6.4]{GR19} (see, also, \cite[Lemma 7.1]{BFR18}) we know that for a point of order greater or equal than $\kappa$, for $\kappa < \tau + \alpha - \theta$, then we have
\begin{equation}
\label{eq.growthest}
\sup_{B_r} |\bar u_\lambda| \le C_M r^{\kappa},
\end{equation}
for some constant $C_M$ depending only on $M$, $\tau$, $\alpha$, $\theta$.

In general, for any point $x_\circ\in \de\Lambda(u_\lambda)$, we can define $\bar u_\lambda^{x_\circ}$ analogously to before as follows.
\begin{defi}
\label{defi.baru}
Let $x_\circ \in \de\Lambda(u_\lambda)$. We define,
\begin{equation}
\label{eq.baru}
\bar u_\lambda^{x_\circ}(x) = u_\lambda(x'+x_\circ', x_{n+1}) -  Q_\tau^{a,x_\circ}(x', x_{n+1}) + Q^{x_\circ}_\tau(x') - \varphi(x'+x_\circ'),
\end{equation}
where $Q^{x_\circ}_\tau(x')$ is the Taylor expansion of order $\tau$ of $\varphi(x_\circ'+x')$, and $Q^{a, x_\circ}_\tau(x')$ is its unique even harmonic extension to $\R^{n+1}$.
\end{defi}

(Notice that, on the thin space, $\bar u_\lambda^{x_\circ}(x', 0) = \bar u_\lambda(x'+x_\circ', 0) $, but this is not true outside the thin space.) Then, $\bar u_\lambda^{x_\circ}(x)$ solves a zero obstacle problem with a right-hand side in ${B_{1-|x_\circ|}}$ (in fact, in $x_\circ + B_1$). With this, we can define the free boundary points of $u_\lambda$ of order $\kappa$, with $1+s \le \kappa < \tau+\alpha-\theta$, as
\[
\Gamma_{\kappa}^\lambda := \{ x_\circ \in \de\Lambda(u_\lambda) : \Phi_{\tau, \alpha, \theta} (0^+, \bar u_\lambda^{x_\circ}) = n+1-2s+2\kappa\},
\]
and similarly
\[
\Gamma_{\ge \kappa}^\lambda :=\{ x_\circ \in \de\Lambda(u_\lambda) : \Phi_{\tau, \alpha, \theta} (0^+, \bar u_\lambda^{x_\circ}) \ge n+1-2s+2\kappa\}.
\]
Equivalently, one can define $\Gamma^\lambda_{\ge \kappa}$ as those points where \eqref{eq.growthest} occurs.

Notice that the previous sets are consistently defined, in the sense that if $x_\circ$ is a free boundary point for $u_\lambda$, and $\tau'\in \N$, $\alpha'\in (0, 1)$ are such that $\tau'+\alpha' \le \tau +\alpha$, then
\[
\Phi_{\tau', \alpha', \theta} (0^+, \bar u_\lambda^{x_\circ}) = \min\bigg\{\Phi_{\tau, \alpha, \theta} (0^+, \bar u_\lambda^{x_\circ}), n+1-2s+2(\tau'+\alpha'-\theta)\bigg\},
\]
(cf. \cite[Lemma 2.3.1]{GP09}), i.e., the definition of free boundary points of order $\kappa$ does not depend on which regularity of the obstacle we consider. In particular, for $C^\infty$ obstacles we can define the points of infinite order as
\begin{equation}
\label{eq.Ginf}
\Gamma_{\infty}^\lambda := \bigcap_{\kappa \ge 2} \Gamma_{\ge\kappa}^\lambda.
\end{equation}

We will need the following lemma, similar to \cite[Lemma 4]{ACS08} and analogous to \cite[Lemma 7.2]{CSS08}.
\begin{lem}
\label{lem.epslem}
Let $w\in C^0(B_1)$, and let $\Lambda\subset B_1\cap \{x_{n+1} = 0\}$. There exists some $\eps_\circ > 0$, depending only on $n$ and $a$, such that if $0<\eps < \eps_\circ$ and
\[
\left\{
\begin{array}{rcll}
w& \ge & 1&\quad \textrm{in}\quad B_1 \cap\{|x_{n+1}|\ge \eps\}\\
w& \ge & -\eps&\quad \textrm{in}\quad B_1\\
|L_a w| & \le&  \eps|x_{n+1}|^a&\quad \textrm{in}\quad B_1\setminus\Lambda\\
w & \ge & 0& \quad\textrm{on}\quad \Lambda,
\end{array}
\right.
\]
then $w > 0$ in $B_{1/2}$.
\end{lem}
\begin{proof}
Suppose that it is not true. In particular, suppose that there exists some $z= (z', z_{n+1}) \in B_{1/2}\setminus \{x_{n+1} = 0\}$ such that $w(z) = 0$.
Let us define the cylinder
\[
Q := \left\{x = (x', x_{n+1}) \in B_1 : |x'-z'|<\frac12, ~~|x_{n+1}-z_{n+1}|<\frac{\sqrt{1+a}}{4} \right\},
\]
and let
\[
P(x) = P(x', x_{n+1}) := |x'-z'|^2- \frac{n}{1+a}x_{n+1}^2
\]
so that $L_a P = 0$. Let
\[
v(x) := w(x) +\frac1n P(x) - \frac{\eps}{1+a} x_{n+1}^2.
\]
Notice that $v(z) =-\frac{n}{n(1+a)} z_{n+1}^2 - \frac{\eps}{1+a} z_{n+1}^2 < 0$. We also have that
\[
L_a v = L_a w - 2 \eps |x_{n+1}|^a \le -\eps |x_{n+1}|^a< 0\quad\textrm{in}\quad B_1\setminus \Lambda,
\]
and
\[
v\ge 0\quad\textrm{on}\quad \Lambda.
\]
That is, $v$ is super- $a$-harmonic and is negative at $z\in Q$, then it must be negative somewhere on $\de Q$. Let us check that this is not the case, to reach a contradiction.

First, notice that, assuming $\eps_\circ < \frac{\sqrt{1+a}}{4}$, on $\de Q\cap \{|x_{n+1}|\ge \eps\}$ we have
\[
v\ge 1-\frac{n}{16(n+1)}-\frac{\eps}{16}\ge 0.
\]

On the other hand, on $\left\{|x'-z'| = \frac12\right\}\cap \{|x_{n+1}|\le \eps\}$ we have
\[
v \ge -\eps +\frac{1}{{n+1}}\left(\frac14 - \frac{n}{1+a}\eps^2\right)-\frac{\eps^3}{1+a} > 0,
\]
if $\eps$ is small enough depending only on $n$ and $a$. Thus, $v \ge 0$ on $\de Q$ and on $\Lambda$, and is super- $a$-harmonic in $Q\setminus \Lambda$, so we must have $v  \ge 0$ in $Q$, contradicting $v(z) < 0$.
\end{proof}

Let us now show the following proposition.

\begin{prop}
\label{prop.kap}
Let $u_\lambda$ satisfy \eqref{eq.thinobst_lam_vp}-\eqref{eq.uassump}, and let $\varphi$ satisfy \eqref{eq.obst}. Let $h > 0$ small, and let $x_\circ \in B_{1-h}\cap\Gamma_{\ge\kappa}^\lambda$ with $\kappa \le \tau + \alpha -a$ and $\kappa < \tau + \alpha$. Then,
\[
u_{\lambda + C_*r^{\kappa-2s}} > \varphi\quad\textrm{in}\quad B_r'(x_\circ'),\quad\textrm{for all}\quad r< \frac{h}{4},
\]
for some $C_*$ depending only on $n$, $s$, $M$, $\kappa$, $\tau$, $\alpha$, and $h$.	

In particular, if $x_\circ \in B_{1-h}\cap \Gamma^\lambda$, then
\begin{equation}
\label{eq.part32}
u_{\lambda + C_*r^{1-s}} > \varphi\quad\textrm{in}\quad B_r'(x_\circ'),\quad\textrm{for all}\quad r< \frac{h}{4},
\end{equation}
for some $C_*$ depending only on $n$, $s$, $M$, $\kappa$, $\tau$, $\alpha$,  and $h$.
\end{prop}
\begin{proof}
Let us assume that $r< \frac{h}{4}$, and let us establish some properties of $\bar u^{x_\circ}_{\lambda+C_*r^{\kappa-2s}}$ in $B_r(0)$ (see Definition~\ref{defi.baru}), for $C_*$ yet to be chosen.

From Lemma~\ref{lem.ulambda} we know that, for any $z\in {B_{h/2}}$,
\begin{align*}
\frac{\bar u^{x_\circ}_{\lambda+\eps}(z) - \bar u^{x_\circ}_\lambda(z)}{\eps} & = \frac{u_{\lambda+\eps}(x_\circ+z) - u_\lambda(x_\circ+z)}{\eps} \\
& \ge c~{\rm dist}^{2s}(x_\circ+z, \Lambda(u_\lambda))\\
& = c~{\rm dist}^{2s}(z, \Lambda(\bar u^{x_\circ}_\lambda)).
\end{align*}
From the previous inequality applied at $x\in B_r(0)\cap \{|x_{n+1}|\ge r \sigma\}$, for some $\sigma > 0$ to be chosen, for $r < \frac{h}{4}$, and with $\eps = C_*r^{\kappa-2s}$ for some $C_*$ to be chosen,
\[
\bar u^{x_\circ}_{\lambda+C_*r^{\kappa-2s}}(x) \ge c\,C_*r^{\kappa-2s} (r\sigma)^{2s} + \bar u^{x_\circ}_\lambda(x)\quad\textrm{for}\quad x \in B_r(0)\cap \{|x_{n+1}|\ge r\sigma\}.
\]
On the other hand, notice that $0$ is a free boundary point of $\bar u^{x_\circ}_\lambda$ of order greater or equal than $\kappa$. In particular, from the growth estimate \eqref{eq.growthest}, we know that
\[
\bar u^{x_\circ}_\lambda \ge -Cr^{\kappa} \quad\textrm{in}\quad B_r(0),\quad \textrm{for}\quad r< \frac{h}{4},
\]
for some $C$ depending only on $n$,  $M$, $s$, $\tau$, $\alpha$, $\theta$, and $h$. By choosing, for example, $\theta = \min\{\frac{\alpha}{2}, \frac{\tau+\alpha-\kappa}{2}\}$ in the definition of the generalized frequency function, \eqref{eq.genPhi}, we can get rid of the dependence on $\theta$. That is,
\[
\bar  u^{x_\circ}_{\lambda+C_*r^{\kappa-2s}}(x) \ge c\,C_*r^\kappa\sigma^{2s} -Cr^{\kappa}\quad\textrm{for}\quad x \in B_r(0)\cap \{|x_{n+1}|\ge r\sigma\}.
\]
Moreover, since $\bar u^{x_\circ}_{\lambda+C_*r^{\kappa-2s}} \ge \bar u^{x_\circ}_{\lambda}$,
\[
\bar u^{x_\circ}_{\lambda+C_*r^{\kappa-2s}} \ge -Cr^{\kappa} \quad\textrm{in}\quad B_r(0),\quad \textrm{for}\quad r< \frac{h}{4}.
\]
Notice, also, that
\[
|L_a \bar u^{x_\circ}_{\lambda+C_*r^{\kappa-2s}}|\le M|x_{n+1}|^a r^{\tau +\alpha-2}\quad\textrm{in}\quad B_r(0)\setminus \Lambda(\bar u^{x_\circ}_{\lambda+C_*r^{\kappa-2s}}).
\]
Let us rescale in domain. We denote
\[
w(x) := \bar u^{x_\circ}_{\lambda+C_* r^{\kappa-2s}} (rx).
\]
Then $w$ is a solution to a thin obstacle problem with right-hand side and with zero obstacle in the ball $B_1$, such that
\[
\left\{
\begin{array}{rcll}
w& \ge & (c\,C_*\sigma^{2s} - C)r^\kappa&\quad \textrm{in}\quad B_1(0) \cap\{|x_{n+1}|\ge \sigma\}\\
w& \ge & -C r^\kappa&\quad \textrm{in}\quad B_1(0)\\
|L_a w| & \le&  M|x_{n+1}|^ar^{\tau + \alpha -a}&\quad \textrm{in}\quad B_1\setminus(\{ x_{n+1} = 0\}\cap \{w = 0\}).
\end{array}
\right.
\]
In particular, if we take $\tilde w := \frac{w}{(c\,C_*\sigma^{2s} - C)r^\kappa}$, then
\[
\left\{
\begin{array}{rcll}
\tilde w& \ge & 1&\quad \textrm{in}\quad B_1(0) \cap\{|x_{n+1}|\ge \sigma\}\\
 \tilde w& \ge & -\frac{C}{c\,C_*\sigma^{2s} - C} &\quad \textrm{in}\quad B_1(0)\\
|L_a \tilde w| & \le&  \frac{M}{c\,C_*\sigma^{2s} - C} |x_{n+1}|^a r^{\tau +\alpha-a -\kappa}&\quad \textrm{in}\quad B_1\setminus(\{ x_{n+1} = 0\}\cap \{ \tilde w = 0\}).
\end{array}
\right.
\]
(Notice that $\tau + \alpha - a -\kappa\ge 0$ by assumption.)
We now want to apply Lemma~\ref{lem.epslem}. We need to choose $\sigma < \eps_\circ(n, a)$, and $C_*$ such that
\[
\frac{C}{c\,C_*\sigma^{2s} - C} < \eps_\circ,\qquad \frac{M}{c\,C_*\sigma^{2s} - C}< \eps_\circ.
\]
By choosing $C_*\gg \eps_\circ^{-1-2s}$ we get that such $C_*$ exists independently of $r$, depending only on $n$, $M$, $s$, $\kappa$, $\tau$, $\alpha$, and $h$.

From Lemma~\ref{lem.epslem}, we deduce that $\tilde w > 0$ in $B_{1/2}$, so that $\bar u^{x_\circ}_{\lambda+C_* r^{\kappa-2s}} > 0$ in $B_{r/2}(0)$. Since $r<h/4$,  we get the desired result, noticing that $\bar u^{x_\circ}_{\lambda+C_* r^{\kappa-2s}}  = (u_{\lambda+C_* r^{\kappa-2s}} -\varphi)(\cdot + x_\circ)$ on $B_r'$.

Finally, notice that thanks to the optimal regularity of solutions, if $x_\circ \in \Gamma^\lambda$, then $x_\circ \in \Gamma^\lambda_{\ge 1+s}$, so that applying the previous result we are done.
\end{proof}

The following corollary will be useful below.
\begin{cor}
\label{cor.u1u2}
Let $u^{(1)}$ and $u^{(2)}$ denote two solutions to
\begin{equation}
  \left\{ \begin{array}{rcll}
 L_a u^{(i)}&=&0 & \textrm{ in } B_1\setminus \left(\{x_{{n+1}} = 0\}\cap \{u^{(i)} = \varphi\}\right)\\
  L_a u^{(i)}&\le&0 & \textrm{ in } {B_1}\\
  u^{(i)} &\ge& \varphi & \textrm{ on } \{x_{{n+1}} = 0\},
  \end{array}\right.\quad\textrm{ for }\quad i \in \{1, 2\}.
\end{equation}
Then, for any $\eps_\circ > 0$ and $h > 0$, there exists a $\delta > 0$ such that if
\[
u^{(2)}\ge u^{(1)},\quad\textrm{and}\quad u^{(2)}\ge u^{(1)} + \eps_\circ\quad\textrm{in}\quad \{|x_{n+1}|\ > 1/2\},
\]
then
\[
\inf\bigg\{ |x_1 - x_2| : x_1\in \de\Lambda(u^{(1)})\cap B_{1-h}, x_2\in \de\Lambda(u^{(2)})\cap B_{1-h}\bigg\}\ge \delta.
\]
\end{cor}
\begin{proof}
The proof follows by Proposition~\ref{prop.kap}. Let us denote $u^{(1)}_\lambda$ the solution to the thin obstacle problem \eqref{eq.thinobst_lam_vp} with boundary data equal to $u^{(1)}$ on $\de B_1\cap \{|x_{n+1}|\le 1/2\}$, and $u^{(1)}_\lambda + \lambda\eps_\circ$ on $\de B_1\cap \{|x_{n+1}|> 1/2\}$. In particular, $u^{(1)}= u^{(1)}_0\le u^{(1)}_1 \le u^{(2)}$.  Moreover, thanks to the Harnack inequality we know that $u^{(1)}_{\lambda+ \eps} \ge u^{(1)}_\lambda + c\eps\eps_\circ$ for $\eps > 0$ in $B_1\cap \{|x_{n+1}|\ge \frac12\}$, for some constant $c$. Thus, if we define
\[
w_\lambda := (c\eps_\circ)^{-1} u^{(1)}_\lambda ,
\]
then $w_\lambda$ fulfil \eqref{eq.uassump}. The result now follows applying Proposition~\ref{prop.kap} to $w_\lambda$ and using that $u^{(1)}= c\eps_\circ w_0 \le c\eps_\circ w_\lambda \le u^{(2)}$ for $\lambda \in [0, 1]$.
\end{proof}

As a direct consequence of Proposition~\ref{prop.kap} (in particular, of \eqref{eq.part32}), we get that if $0\in \de\Lambda(u_\lambda)$, then $0\notin \de\Lambda(u_{\bar\lambda})$ for $\bar \lambda \neq \lambda$ (since $u_{\lambda+C_*\delta^{1-s}} > \varphi$ in $B_\delta$ for $\delta > 0$ small enough).

In particular:
\begin{defi} We define
\[
\Gamma_{\kappa} := \bigcup_{\lambda\in [0, 1]} \Gamma^\lambda_{\kappa},\qquad \Gamma_{\ge \kappa} := \bigcup_{\lambda\in [0, 1]} \Gamma^\lambda_{\ge \kappa},\qquad\textrm{and}\qquad\Gamma:= \bigcup_{\lambda\in [0, 1]} \Gamma^\lambda.
\]
We also define
\begin{equation}
\label{eq.lamcirc}
\lambda (x_\circ) := \big\{\lambda \in [0, 1]: x_\circ\in  \de\Lambda(u_\lambda)\big\},
\end{equation}
which is uniquely defined on $\Gamma$.
\end{defi}

The fact that $\lambda(x_\circ)$ is uniquely defined for $x_\circ \in \Gamma$ follows since $\Gamma_{\kappa}\cap \Gamma_{\bar\kappa} = \varnothing$ if $\kappa\neq \bar\kappa$. In particular, if $x_\circ\in \Gamma_{\kappa}$ then $x_\circ \in \Gamma^{\lambda(x_\circ)} = \de\Lambda(u_{\lambda(x_\circ)})$.

A direct consequence of Proposition~\ref{prop.kap} is that $\Gamma \ni x_\circ \mapsto \lambda(x_\circ)$ is continuous:

\begin{cor}
\label{cor.lamcont}
Let $u_\lambda$ satisfy \eqref{eq.thinobst_lam_vp}-\eqref{eq.uassump}, and let $\varphi$ satisfy \eqref{eq.obst}. The function
\[\Gamma \ni x_\circ \mapsto \lambda(x_\circ)\]
for $\lambda(x_\circ)$ defined by \eqref{eq.lamcirc} is continuous. Moreover, for each $h > 0$,
\[\Gamma \cap B_{1-h}\ni x_\circ \mapsto \bar u^{x_\circ}_{\lambda(x_\circ)}\]
is continuous in the $C^0$-norm.
\end{cor}
\begin{proof}
Let us start with the first statement. If $x_1,x_2 \in \Gamma$ are such that $|x_1-x_2|\le\frac\delta2$ for $\delta >0$ small enough, and $\lambda(x_1)\ge \lambda(x_2)$, then
\[
u_{\lambda(x_2)+C_*\delta^{1-s}} > \varphi\quad\textrm{in}\quad B_\delta(x_\circ)
\]
by Proposition~\ref{prop.kap}. In particular, $\lambda(y) < \lambda(x_2)+C_*\delta^{1-s}$ for any $y\in B_\delta(x_2)$, so that $\lambda(x_1)<\lambda(x_2)+C_*\delta^{1-s}$. That is,
\[
|\lambda(x_1)-\lambda(x_2)|\le C_*\delta^{1-s}
\]
and $\lambda(x)$ is continuous (in fact, it is $(1-s)$-H\"older continuous).

Let us now show that
\[\Gamma \cap B_{1-h} \ni x_\circ \mapsto \bar u^{x_\circ}_{\lambda(x_\circ)}\]
is also continuous (in the $C^0$-norm).
From the definition of $\bar u^{x_\circ}_{\lambda(x_\circ)}$, Definition~\ref{defi.baru}, and since $\varphi$ is continuous, it is enough to show that $\Gamma\cap B_{1-h} \ni x_\circ\mapsto u_{\lambda(x_\circ)}(x_\circ+\,\cdot)$ is continuous. Moreover, since each $u_\lambda$ is continuous (and in fact, they are uniformly $C^{2s}$), we will show that $\Gamma\ni x_\circ\mapsto u_{\lambda(x_\circ)}$ is continuous, in the sense that, for every $\eps > 0$, there exists some $\delta > 0$ such that if $x, z\in \Gamma\cap B_{1-h} $ (for some $h > 0$), $|x-z|\le \delta$, then
\[
\sup_{B_1} |u_{\lambda(x)} - u_{\lambda(z)}|\le \eps.
\]

Let us argue by contradiction. Suppose that it is not true, and that there exist sequences $x_i, z_i \in B_{1-h}\cap \Gamma$ such that $|x_i - z_i|\le \frac{1}{i}$ and
\[
\sup_{B_1} |u_{\lambda(x_i)} - u_{\lambda(z_i)}|\ge \eps_\circ > 0,
\]
for some $\eps_\circ > 0$. In particular, let us assume that $\lambda(x_i) > \lambda(z_i)$, so that $u_{\lambda(x_i)} \ge u_{\lambda(z_i)}$. After taking a subsequence (by compactness, using also that $\|u_\lambda\|_{C^{2s}(B_1)}\le M$), we can assume that there exists some ball $B_\rho(y)\subset B_1$ such that
\[
u_{\lambda(x_i)} \ge u_{\lambda(z_i)} + \frac{\eps_\circ}{2}\quad\textrm{in}\quad B_\rho(y)\subset B_1
\]
for all $i\in \N$. (The radius $\rho$ depends only on $n$, $\eps_\circ$, and $M$.) By interior Harnack's inequality, we have that
\[
u_{\lambda(x_i)} \ge u_{\lambda(z_i)} + c\frac{\eps_\circ}{2}\quad\textrm{in}\quad B_{h/2}(z_i)\cap \{|x_{n+1}|\ge h/4\},
\]
for some constant $c$ depending on $\rho$ and $h$. After translating and scaling, we are in a situation to apply Corollary~\ref{cor.u1u2}. In particular, for some $\delta > 0$ (depending on $\eps_\circ$ and $h$), $|x_i - z_i|\ge \delta > 0$. This is a contradiction with $|x_i -z_i|\le \frac{1}{i}$ for $i \in \N$ large enough. Therefore, $x_\circ \mapsto \bar u^{x_\circ}_{\lambda(x_\circ)}$ is continuous.
\end{proof}

The following lemma improves Lemma~\ref{lem.ulambda} in case $x_\circ\in \Gamma_2$. We denote here $a_- := \max\{0, -a\}$.

\begin{lem}
\label{lem.ulambda_theta}
Let $u_\lambda$ satisfy \eqref{eq.thinobst_lam_vp}-\eqref{eq.uassump}, and let $\varphi$ satisfy \eqref{eq.obst}. Let $n \ge 2$, and $h >0$ small.  Let $x_\circ\in B_{1-h}\cap \Gamma^\lambda_{2}$. Then, for each $\eta >0 $ small, and for $\mu > \lambda$,
\begin{enumerate}[label=(\roman*)]
 \item if $s \ge \frac12$,
\[
\de^+_\lambda \bar u^{x_\circ}_{\mu}(0)= \de^+_\lambda u_{\mu}(x_\circ) \ge c~{\rm dist}^{\eta+a_-}(x_\circ, \Lambda( u_\mu))  =c~{\rm dist}^{\eta-a}(0, \Lambda( \bar u^{x_\circ}_\mu)),
\]
\item if $s \le \frac12$,
\[
\de^+_\lambda \bar u^{x_\circ}_{\mu}(0)= \de^+_\lambda u_{\mu}(x_\circ) \ge c~{\rm dist}^{\eta+a_-}(x_\circ, \Lambda( u_\mu))  =c~{\rm dist}^{\eta}(0, \Lambda( \bar u^{x_\circ}_\mu)),
\]
\end{enumerate}
for some constant $c >0$ independent of $\lambda$ and $\mu$ (but possibly depending on everything else).
\end{lem}
\begin{proof}
Fix some $\mu > 0$ and $\eps >0 $ small, and define
\[
\delta_{\lambda, \eps} \bar u^{x_\circ}_\mu(x) = \frac{\bar u^{x_\circ}_{\mu +\eps}(x) - \bar u^{x_\circ}_{\mu}(x)}{\eps} = \frac{u_{\mu +\eps}(x+x_\circ) - u_{\mu}(x+x_\circ)}{\eps} .
\]
As in the proof of Lemma~\ref{lem.ulambda}, we know that $ \delta_{\lambda, \eps}\bar u^{x_\circ}_\mu (x)\ge 0$, $\delta_{\lambda,\eps} \bar u^{x_\circ}_\mu \ge 1$ on $(-x_\circ + \de{B_1})\cap\{|x_{n+1}|\ge \frac12\}$, and
\begin{equation}
\label{eq.deltharm}
L_a \delta_{\lambda, \eps} \bar u^{x_\circ}_\mu = 0\quad\textrm{in}\quad (-x_\circ+{B_1})\setminus \Lambda(\bar u^{x_\circ}_\mu)\supset (-x_\circ+{B_1})\setminus \Lambda(\bar u^{x_\circ}_\lambda).
\end{equation}

Let us start by showing that, for every $A > 0$, there exists some $\rho_A > 0$ (independent of $\mu$) such that, after a rotation,
\begin{equation}
\label{eq.LA}
\Lambda(\bar u^{x_\circ}_\mu)\cap B_{\rho_A} \subset \{|x'|^2 \ge A x_1^2\}.
\end{equation}
In particular, we will show that, for every $A > 0$, there exists some $\rho_A > 0$ such that, after a rotation,
\begin{equation}
\label{eq.LA_lam}
\Lambda(\bar u^{x_\circ}_\lambda)\cap B_{\rho_A} \subset \{|x'|^2 \ge A x_1^2\}.
\end{equation}
(Notice that now we have taken $\mu\downarrow \lambda$, and since the contact set is decreasing in $\lambda$, \eqref{eq.LA_lam} implies \eqref{eq.LA}.)

Indeed, by \cite[Theorem 8.2]{GR19}, we know that
\[
\bar u^{x_\circ}_\lambda(x) = p_2(x) + o(|x|^2)
\]
for some 2-homogeneous, $a$-harmonic polynomial, such that $p_2\ge 0$ on $\{x_{n+1} = 0\}$ (recall that we are assuming that $x_\circ \in \Gamma^\lambda_{2}$) and $p_2\not\equiv 0$. After a rotation, thus, we may assume that $p_2(x', 0) \ge c x_1^2$. That is,
\[
\bar u^{x_\circ}_\lambda(x', 0) \ge cx_1^2 + o(|x'|^2)> \frac{c}{A}|x'|^2 + o(|x'|^2) > 0\quad\textrm{in}\quad B_{\rho_A} \cap  \{|x'|^2 < A x_1^2\}
\]
if $\rho_A$ is small enough (depending on $A$, but also on the point $x_\circ$, and the function $\bar u^{x_\circ}_\lambda$). That is, \eqref{eq.LA_lam}, and in particular, \eqref{eq.LA}, holds. Considering again the $x_{n+1}$ direction, we know that for every $A> 0$ there exists some $\rho_A$ such that, after a rotation,
\begin{equation}
\label{eq.LA2}
\Lambda(\bar u^{x_\circ}_\mu)\cap B_{\rho_A} \subset \{x_1^2 + x_{n+1}^2\le A^{-1}|x'|^2 \} =: \mathcal{C}_A.
\end{equation}
Notice that $\rho_A \downarrow 0$ as $A\to \infty$. Let us suppose that we are always in the rotated setting so that the previous inclusion holds. Let us denote $\psi_A$ the unique homogeneous solution to
\[
\left\{
\begin{array}{rcll}
L_a \psi_A & = & 0\quad\textrm{in}\quad \R^n\setminus \mathcal{C}_{A/2}\\
\psi_A & = & 0\quad\textrm{in}\quad \mathcal{C}_{A/2}\\
\psi_A & \ge & 0\quad\textrm{in}\quad \R^n,
\end{array}
\right.
\]
such that $\sup_{\de B_1} \psi_A = 1$.

Let  $\eta_\circ > 0$ denote the homogeneity of $\psi_A$ (i.e., $\psi_A(t x) = t^{\eta_\circ}\psi_A(x)$). It corresponds to the first eigenvalue on the sphere $\mathbb{S}^{n}$ of $L_a$ with zero boundary condition on $\mathcal{C}_{A/2}$. Alternatively, it corresponds to the infimum of the corresponding Rayleigh quotient among functions with the same boundary values. Notice that, as $A\to \infty$, $\mathcal{C}_{A/2}\to \{x_1 = x_{n+1} = 0\}$ locally uniformly in the Hausdorff distance, and $\{x_1 = x_{n+1} = 0\}$ has zero $a$-harmonic capacity when $s\le \frac12$ (see \cite[Corollary 2.12]{Kil94}). Thus, when $s\le \frac12$ the infimum of the Rayleigh quotient converges to the first eigenvalue of $L_a$ on the sphere without boundary conditions (namely, 0), and thus, $\eta_\circ \downarrow 0$ as $A\to \infty$ if $a\ge 0$. Alternatively, if $s > \frac12$ the first eigenvalue corresponds to the homogeneity $-a$ (attained by the function $(x_1^2 + x_{n+1}^2)^{-a/2}$), so that $\eta_\circ\downarrow -a$ as $A\to \infty$ if $a<0$. In all, $\eta_\circ\downarrow a_-$, with $a_- =\max\{0, -a\}$.

Let us choose some $A$ large enough such that $\eta_\circ < \eta+a_-$.
Now, let
\[
r:= {\rm dist}(x_\circ, \Lambda(u_\mu)) = {\rm dist}(0, \Lambda(\bar u^{x_\circ}_\mu)),
\]
and let $\psi_{A, r}$ for $r < {\rho_A}/2$ denote the solution to
\[
\left\{
\begin{array}{rcll}
L_a \psi_{A, r} & = & 0&\quad\textrm{in}\quad B_r\cup \left(B_{\rho_A/2} \setminus \mathcal{C}_{A/2}\right)\\
\psi_{A,r} & = & 0&\quad\textrm{in}\quad (B_{\rho_A/2}\cap \mathcal{C}_{A/2})\setminus B_r\\
\psi_{A,r} & = & \psi_A &\quad\textrm{on}\quad \de B_{\rho_A/2}.
\end{array}
\right.
\]

Let $\bar c$ small enough (depending on $\rho_A$, $A$, $h$, $n$, $s$, $M$) such that $\bar c\psi_A \le \delta_{\lambda, \eps} \bar u^{x_\circ}_\mu$ on $\de B_{\rho_A/2}$. For instance, take
\[
\bar c = \inf_{x\in \de B_{\rho_A/2} \cap \mathcal{C}_{A/2}^c} \delta_{\lambda, \eps}\bar u^{x_\circ}_\mu (x) > 0,
\]
which is positive since $\delta_{\lambda, \eps} u_\mu \ge 0$, $\delta_{\lambda, \eps}u_\mu \ge 1 $ on $\de B_1\cap\{|x_{n+1}| = 0\}$, and $L_a \delta_{\lambda, \eps}  u_\mu = 0$ in $(B_1\setminus\{x_{n+1} = 0\})\cup (B_{\rho_A}(x_\circ) \setminus \mathcal{C}_A)$ (recall $\delta_{\lambda, \eps} u_\mu = \delta_{\lambda, \eps} \bar u^{x_\circ}_\mu(\cdot - x_\circ)$), and thus, by strong maximum principle (or Harnack's inequality, see \cite[Theorem 2.3.8]{FKS82}) we must have $\bar c > 0$ depending only on $\rho_A$, $A$, $h$, $n$, $s$, $M$.

Now notice that $\bar c\psi_{A,r} \le \delta_{\lambda, \eps} \bar u^{x_\circ}_\mu$ on $\de B_{\rho_A/2}$, $\bar c\psi_{A,r} \le \delta_{\lambda, \eps} \bar u^{x_\circ}_\mu$ on $B_{\rho_A/2}\cap  \mathcal{C}_{A/2} \setminus B_r$, and both $\bar c\psi_{A,r}$ and $\delta_{\lambda, \eps} \bar u^{x_\circ}_\mu$ are $a$-harmonic in $B_r\cup (B_{\rho_A/2}\setminus \mathcal{C}_{A/2})$ (thanks to \eqref{eq.deltharm}-\eqref{eq.LA2}). By comparison principle
\[
\bar c\psi_{A}\le \bar c\psi_{A,r} \le \delta_{\lambda, \eps} \bar u^{x_\circ}_\mu\quad\textrm{in}\quad B_{\rho_A/2}.
\]

By Harnack's inequality, there exists a constant $C$ depending only on $n$ and $s$ such that
\[
\psi_{A,r}(0) \ge \inf_{B_{r/2}(0)} \psi_{A,r} \ge \frac{1}{C}\sup_{B_{r/2}(0)} \psi_{A,r} \ge  \frac{1}{C}\sup_{B_{r/2}(0)} \psi_{A} \ge cr^{\eta_\circ},
\]
where in the last inequality we are using the $\eta_\circ$-homogeneity of $\psi_A$, and $c$ depends only on $n$ and $a$. Thus,
\[
\delta_{\lambda, \eps} \bar u^{x_\circ}_\mu (0) \ge \bar c \psi_{A,r}(0) \ge c \bar c r^{\eta_\circ} = c~{\rm dist}^{\eta_\circ}(x_\circ, \Lambda(u_\mu)) = c~{\rm dist}^{\eta_\circ}(0, \Lambda(\bar u^{x_\circ}_\mu)),
\]
for some $c > 0$ that might depends on everything, but it is independent of $\mu$ and $\lambda$, where we assumed $r < \rho_A/2$. We can reach all $r > 0$ by taking a smaller $c > 0$ (independent of $\lambda$ and $\mu$), thanks to Lemma~\ref{lem.ulambda}. Recalling $\eta_\circ < \eta+a_-$, and letting $\eps\downarrow 0$, this gives the desired result.
\end{proof}

Using the previous lemma, combined with an ODE argument, we find the following.

\begin{prop}
\label{prop.2grow}
Let $x_\circ \in \Gamma^\lambda_{2}$ be any point of order $2$. Then,
\begin{itemize}
\item If $s\le \frac12$, for every $\eps_\circ > 0$, there exists some $\delta_\circ > 0$ such that
\[
\Gamma^{\lambda+ \delta^{2-\eps_\circ}}_{2} \cap B_\delta(x_\circ) = \varnothing,
\]
for all $\delta\in (0, \delta_\circ)$.
\item If $s > \frac12$, for every $\eps_\circ > 0$, there exists some $\delta_\circ > 0$ such that
\[
\Gamma^{\lambda+ \delta^{2\frac{2-s}{1+s}-\eps_\circ}}_{2} \cap B_\delta(x_\circ) = \varnothing,
\]
for all $\delta\in (0, \delta_\circ)$.
\end{itemize}
\end{prop}
\begin{proof}
We use Lemma~\ref{lem.ulambda_theta}. We know that, for each $\eta >0$ small,
\[
\de^+_\lambda \bar u^{x_\circ}_{\mu}(0) \ge c~{\rm dist}^{\eta+a_-}(0,  \Lambda(\bar u^{x_\circ}_\mu))\quad\textrm{for}\quad \mu > \lambda.
\]
On the other hand, from the optimal regularity for the thin obstacle problem, we know that
\[
\bar u^{x_\circ}_\mu(0) \le C{\rm dist}^{1+s}(0, \Lambda(\bar u^{x_\circ}_\mu)),
\]
which gives
\[
\de^+_\lambda \bar u^{x_\circ}_{\mu}(0) \ge c(\bar u^{x_\circ}_\mu(0) )^{\frac{\eta+a_-}{1+s}}.
\]
Solving the ODE between $\lambda$ and $\mu$, this yields
\[
\bar u^{x_\circ}_{\mu}(0)^{1-\frac{\eta+a_-}{1+s}}\ge c(\mu-\lambda)\quad\Longleftrightarrow \quad \bar u^{x_\circ}_{\mu}(0)\ge c(\mu-\lambda)^{\frac{2+2s}{3-2\eta-|a|}}.
\]
Let us now suppose that there exists some $z_\circ\in B_\delta(x_\circ)\cap \Gamma^\mu_{2}$. Notice that $\bar u^{z_\circ}_\mu$ has quadratic growth around zero (since $z_\circ$ is a singular point of order 2), that is $\bar u^{z_\circ}_\mu\le C\rho^2$ in $B_\rho'\times\{0\}$ for $\rho > 0$. Thus, using that $\bar u^{x_\circ}_\mu = \bar u^{z_\circ}_\mu(\cdot + x_\circ- z_\circ )$ in $B_1'$
\[
C\delta^2 \ge \bar u^{z_\circ}_\mu(x_\circ- z_\circ ) = \bar u^{x_\circ}_\mu(0) \ge c(\mu-\lambda)^{\frac{2+2s}{3-2\eta-|a|}},
\]
that is, $\mu - \lambda \le C\delta^{\frac{3-2\eta-|a|}{1+s}}$. In particular, whenever $\mu - \lambda > C\delta^{\frac{3-2\eta-|a|}{1+s}}$ then $B_\delta(x_\circ)\cap \Gamma^\mu_{2} = \varnothing$.

Taking $\delta$ and $\eta$ small enough we get the desired result.
\end{proof}

\section{Dimension of $\Gamma_2$}
\label{sec.dimGam2}
In this section we prove that $\Gamma_2  = \bigcup_{\lambda\in [0,1]} \Gamma_2^\lambda$ has dimension at most $n-1$.

\begin{prop}
\label{prop.uniconv}
Let $m\in \N$, and suppose $2m < \tau +\alpha$. Let us denote $p_{2m}^{x_\circ}$ the blow-up of $\bar  u_{\lambda(x_\circ)}^{x_\circ}$ at $x_\circ\in \Gamma_{2m}$. Then, the mapping $\Gamma_{2m}\ni x_\circ\mapsto p_{2m}^{x_\circ}$ is continuous. Moreover, for any compact set $K \subset \Gamma_{2m}$ there exists a modulus of continuity $\sigma_K$ such that
\[
|\bar u_{\lambda(x_\circ)}^{x_\circ}(x) - p_{2m}^{x_\circ}(x)|\le \sigma_K(|x|) |x|^{2m}
\]
for any $x_\circ\in K$.
\end{prop}
\begin{proof}
This follows exactly as the proof of \cite[Theorem 2.8.4]{GP09} (or \cite[Theorem 8.2]{GR19}) using that $\Gamma_{2m}\ni x_\circ\mapsto \lambda(x_\circ)$ and $\Gamma_{2m} \ni x_\circ \mapsto \bar u_{\lambda(x_\circ)}^{x_\circ}$ are continuous (see Corollary~\ref{cor.lamcont}).
\end{proof}

Singular points (that is, points of order $2m < \tau +\alpha$) have a non-degeneracy property. Namely, as proved in \cite[Lemma 8.1]{GR19}, if $x_\circ \in \Gamma_{2m}^\lambda$, then there exists some constant $C > 0$ (depending on the point $x_\circ$) such that
\[
C^{-1} r^{2m} \le \sup_{\de B_r} |\bar u_\lambda^{x_\circ}|\le C r^{2m}.
\]
In particular, we can further divide the set $\Gamma_{2m}$ according to the degree of degeneracy of the singular point. That is, let us define
\[
\Gamma_{2m, j} := \{x_\circ \in B_{1-j^{-1}} \cap \Gamma_{2m} : j^{-1} r^{2m} \le \sup_{\de B_r} |\bar u_{\lambda(x_\circ)}^{x_\circ}|\le j r^{2m} \textrm{ for all } r \le (2j)^{-1}\},
\]
so that
\[
\Gamma_{2m} = \bigcup_{j\in \N} \Gamma_{2m,j},
\]
and each $\Gamma_{2m,j}\subset \Gamma_{2m}$ is compact (see \cite[Lemma 2.8.2]{GP09}, which only uses the upper semi-continuity of the frequency formula with respect to the point).

In the next proposition we are going to use a Monneau-type monotonicity formula. In particular, we will use that, if we define for $m\in \N$, $x_\circ \in \Gamma_{2m}^\lambda$,
\begin{equation}
\label{eq.Monneau1}
\mathcal{M}_m(r, \bar u^{x_\circ}_{\lambda}, p_{2m}) := \frac{1}{r^{n+a+4m}}\int_{\de B_r} (\bar u_\lambda^{x_\circ} - p_{2m})^2|x_{n+1}|^a,
\end{equation}
for any $2m$-homogeneous, $a$-harmonic, even polynomial $p_{2m}$ with $p_{2m}(x', 0)\ge 0$, such that $p_{2m}\le C$ for some universal bound $C$, then
\begin{equation}
\label{eq.Monneau2}
\frac{d}{dr} \mathcal{M}_m(r, \bar u^{x_\circ}_{\lambda}, p_{2m}) \ge -C_M r^{\alpha -1}
\end{equation}
for some constant $C_M$ independent of $\lambda$. (See \cite[Proposition 7.2]{GR19} and \cite[Theorem 2.7.2]{GP09}.)

\begin{prop}
\label{prop.whitapp}
Let $m\in \N$, and suppose $2m < \tau +\alpha$. Let us denote $p_{2m}^{x_\circ}$ the blow-up of $\bar  u_{\lambda(x_\circ)}^{x_\circ}$ at $x_\circ\in \Gamma_{2m}$. Then, for each $j\in \N$ there exists a modulus of continuity $\sigma_j$ such that
\[
\|p_{2m}^{x_\circ} - p_{2m}^{z_\circ}\|_{L^2(\de B_1, |x_{n+1}|^a)}\le \sigma_j (|x_\circ - z_\circ|)
\]
for all $x_\circ, z_\circ \in \Gamma_{2m, j}$.
\end{prop}
\begin{proof}
Suppose it is not true. That is, suppose that there exist sequences $x_k, z_k \in \Gamma_{2m, j}$ with $k\in \N$, such that $|x_k-z_k|\to 0$ and
\begin{equation}
\label{eq.contM}
\|p_{2m}^{x_k} - p_{2m}^{z_k}\|_{L^2(\de B_1, |x_{n+1}|^a)} \ge \delta > 0
\end{equation}
for some $\delta > 0$. Suppose also that $\lambda(x_k)\le \lambda(z_k)$.

Let $\rho_k :=|x_k - z_k|\downarrow 0$ as $k \to \infty$. Let us define
\[
v_x^k(x) := \frac{\bar u^{x_k}_{\lambda(x_k)}(\rho_k x)}{\rho_k^{2m}}
\qquad\textrm{and}\qquad
v_z^k(x):= \frac{\bar u^{z_k}_{\lambda(z_k)}(\rho_k x+x_k-z_k)}{\rho_k^{2m}}.
\]

We have that
\begin{align*}
v_z^k(x) - v_x^k(x) & = \rho_k^{-2m}\big\{u_{\lambda(z_k)}(\rho_k x + x_k) - u_{\lambda(x_k)}(\rho_k x + x_k)+ Q_\tau^{x_k}(\rho_k x')\\
& \quad -Q_\tau^{z_k}(\rho_k x'+x_k'-z_k')-{\rm Ext}_a ( Q_\tau^{x_k}(\rho_k \cdot )-Q_\tau^{z_k}(\rho_k \cdot +x_k'-z_k'))(x', x_{n+1})\big\},
\end{align*}
where, if $p = p(x'):\R^{n}\to \R$ is a polynomial, ${\rm Ext}_a (p) (x', x_{n+1})$ denotes its unique even $a$-harmonic extension.

Notice that $u_{\lambda(z_k)} \ge u_{\lambda(x_k)}$ (since $\lambda(z_k) \ge \lambda(x_k)$). On the other hand, let us study the convergence of the degree $\tau$ polynomials $P_\tau^k (x') = Q_\tau^{x_k}(\rho_k x')-Q_\tau^{z_k}(\rho_k x'+x_k'-z_k')$. First, observe that
\[
|P_\tau^k(0)| = |Q_\tau^{x_k}(0) - Q_\tau^{z_k}(x_k' - z_k')| = |\varphi(x_k') - Q_\tau^{z_k}(x_k' - z_k')| = o(\rho_k^\tau),
\]
since $Q_\tau^{x_k}$ and $Q_\tau^{z_k}$ are the Taylor expansions of $\varphi$ of order $\tau$ at $x_k$ and $z_k$ respectively, and $|x_k'-z_k'| = \rho_k$. Similarly, for any multi-index $\beta = (\beta_1, \dots, \beta_{n-1})$ with $|\beta|\le \tau$,
\[
|D^\beta P_\tau^k(0)| = \rho_k^{|\beta|}\left| D^\beta\varphi(x_k) - D^\beta Q_\tau^{z_k}(x_k'-z_k') \right| = o(\rho_k^\tau).
\]
Thus, the  $P_\tau^k = o(\rho_k^\tau)$ (say, in any norm in $B_1'$), and so the same occurs with the $a$-harmonic extension. Notice, also, that by assumption, $2m\le \tau$. In all, we have that
\begin{equation}
\label{eq.difo1}
v_z^k(x) - v_x^k(x)\ge o(1).
\end{equation}

On the other hand, we have
\begin{equation}
\label{eq.vx}
|v_x^k(x) - p_{2m}^{x_k}(x)|\le \sigma_{K, j}(\rho_k |x|)|x|^{2m}
\end{equation}
thanks to Proposition~\ref{prop.uniconv} with $K = \Gamma_{2m, j}$, and for some modulus of continuity $\sigma_{K, j}$ depending on $j$. Similarly, if we denote
\[
\xi_k = \frac{z_k - x_k}{\rho_k}\in \mathbb{S}^{n},
\]
then
\begin{equation}
\label{eq.vz}
|v_z^k(x) - p_{2m}^{z_k}(x-\xi_k)|\le \sigma_{K, j}(\rho_k|x-\xi_k|)|x-\xi_k|^{2m}.
\end{equation}

From the definition of $\Gamma_{2m, j}$ we know that
\begin{equation}
\label{eq.bound}
j^{-1} r^{2m} \le \sup_{\de B_r} |p_{2m}^{x_k}|\le j r^{2m}.
\end{equation}
In particular, up to subsequences, $p_{2m}^{x_k} \to p_x$ uniformly for some $2m$-homogeneous polynomial $p_x$, $a$-harmonic, such that $p_x(x', 0)\ge 0$, and
\begin{equation}
\label{eq.boundholds}
j^{-1} r^{2m} \le \sup_{\de B_r} |p_x|\le j r^{2m}.
\end{equation}
Notice that both bounds \eqref{eq.bound} are crucial: the bound from above allows a convergence, and the bound from below avoid getting as a limit the zero polynomial. We similarly have that $p_{2m}^{z_k}\to p_z$ for some $p_z$ $2m$-homogeneous polynomial, $a$-harmonic, with $p_z(x', 0) \ge 0$ and such that \eqref{eq.boundholds} holds for $p_z$.

Combining the convergences of $p^{x_k}_{2m}$ and $p_{2m}^{z_k}$ to $p_x$ and $p_z$ with \eqref{eq.vx}-\eqref{eq.vz} we obtain that
\[
v_x^k \to p_x,\quad v_z^k \to p_z(\cdot -\xi_\circ), \quad\textrm{uniformly},
\]
for some $\xi_\circ = (\xi'_\circ, 0) \in \mathbb{S}^{n}$. On the other hand, from \eqref{eq.difo1}, we know that $p_x \ge p_z(\cdot - \xi_\circ)$.

Thus, $p_x - p_z(\cdot - \xi_\circ)\ge 0$, and is $a$-harmonic, therefore by Lioville's theorem is constant. Moreover, both terms are non-negative on the thin space, and both attain the value 0 (since they are homogeneous), therefore, $p_x = p_z(\cdot - \xi_\circ)$. Since both $p_x$ and $p_z$ are homogeneous of the same degree, this implies that $p_x = p_z$.

Let us now use the Monneau-type monotonicity formula, \eqref{eq.Monneau1}-\eqref{eq.Monneau2}, with polynomials $p_x$ and $p_z$:
\begin{align*}
\int_{\de B_1} (v_x^k - p_x)^2|x_{n+1}|^a & = \mathcal{M}_m(\rho_k, \bar u_{\lambda(x_k)}^{x_k}, p_x) \\
& \ge \mathcal{M}_m(0^+, \bar u_{\lambda(x_k)}^{x_k}, p_x) -C_M \rho_k^\alpha\\
& = \int_{\de B_1} (p_{2m}^{x_k} - p_x)^2|x_{n+1}|^a - C_M \rho_k^\alpha,
\end{align*}
where we are using that $\rho^{-2m} \bar u_{\lambda(x_k)} (\rho x) \to p_{2m}^{x_k}$ as $\rho \downarrow 0$. Letting $k\to \infty$ (so $\rho_k \downarrow 0$), since $v_x^k \to p_x$ we get that
\[
\int_{\de B_1} (p_{2m}^{x_k} - p_x)^2|x_{n+1}|^a \to 0.
\]

On the other hand, proceeding analogously,
\[
\int_{\de B_1} (v_z^k(\cdot + \xi_k) - p_z)^2|x_{n+1}|^a \ge  \int_{\de B_1} (p_{2m}^{z_k} - p_z)^2|x_{n+1}|^a - C_M \rho_k^\alpha,
\]
and since $v_z^k \to p_z(\cdot - \xi_\circ)$,
\[
\int_{\de B_1} (p_{2m}^{z_k} - p_z)^2|x_{n+1}|^a \to 0.
\]
Thus, since $p_x = p_z$, we obtain that
\[
\int_{\de B_1} (p_{2m}^{z_k} - p_{2m}^{x_k})^2|x_{n+1}|^a \to 0,
\]
a contradiction with \eqref{eq.contM}.
\end{proof}

Finally, we prove the following.

\begin{prop}
\label{prop.singdim}
Let $m\in \N$, and suppose $2m < \tau + \alpha$. Then, $ \Gamma_{2m}$ is contained in a countable union of $(n-1)$-dimensional $C^1$ manifolds.
\end{prop}
\begin{proof}
The proof is now standard, and it follows applying the Whitney extension theorem, which can be applied thanks to Proposition~\ref{prop.whitapp}. We refer the reader to the proof of \cite[Theorem 1.3.8]{GP09}, which we summarise here for completeness.

Indeed, if $x_\circ\in \Gamma_{2m}$, and $\beta = (\beta_1,\dots, \beta_{n+1})$ is a multi-index, we denote
\[
p^{x_\circ}_{2m}(x) = \sum_{|\beta|=2m} \frac{a_\beta(x_\circ)}{\beta!} x^\beta
\]
so that $a(x_\circ)$ (the coefficients) are continuous on $\Gamma_{2m, j}$ by Proposition~\ref{prop.whitapp}. Arguing as in \cite[Lemma 1.5.6]{GP09} (by means of Proposition~\ref{prop.uniconv}) the function $f_\beta$ defined for the multi-index $\beta$, with $|\beta|\le 2m$,
\[
f_\beta(x) = \left\{
\begin{array}{ll}
0 & \textrm{if }|\beta|< 2m,\\
a_\beta(x) & \textrm{if }|\beta| = 2m,
\end{array}
\right.
\]
for $x\in \Gamma_{2m}$, fulfils the compatibility conditions to apply Whitney's extension theorem on $\Gamma_{2m, j}$. That is, there exists some $F\in C^{2m}(\R^{n+1})$ such that
\[
\frac{d^{|\beta|}}{dx^\beta} F  = f_\beta\quad\textrm{ on }\quad \Gamma_{2m, j},
\]
for any $|\beta|\le 2m$.

Now, for any $x_\circ\in \Gamma_{2m, j}$, since $p_{2m}^{x_\circ}\neq 0$, there exists some $\nu \in \R^n$ such that
\[
\nu \cdot \nabla_{x'} p_{2m}^{x_\circ}(x', 0) \neq 0\quad\textrm{ on }\quad \R^n.
\]
In particular, for some multi-index $\beta_\circ$ with $|\beta_\circ| = 2m-1$,
\begin{equation}
\label{eq.Feq}
\nu \cdot \nabla_{x'} \de^{\beta_\circ} F(x_\circ) = \nu \cdot \nabla_{x'} \de^{\beta_\circ} p_{2m}^{x_\circ} (0) \neq 0,
\end{equation}
where $\de^{\beta_\circ} := \frac{d^{|\beta_\circ|}}{dx^{\beta_\circ}}$. On the other hand,
\[
\Gamma_{2m, j}\subset \bigcap_{|\beta| = 2m-1} \{\de^\beta F = 0\} \subset \{\de^{\beta_\circ} F = 0\},
\]
so that, thanks to \eqref{eq.Feq}, by the implicit function theorem $\Gamma_{2m, j}$ is locally contained in a $(n-1)$-dimensional $C^1$ manifold. Thus, $\Gamma_{2m}$ is contained in a countable union of $(n-1)$-dimensional $C^1$ manifolds.
\end{proof}

\section{Proof of main results}
\label{sec.mainresults}
Finally, in this section we prove the main results. To do so, the starting point is the following GMT lemma from \cite{FRS19}.

\begin{lem}[\cite{FRS19}]
\label{lem.CL}
Consider the family $\{E_\lambda\}_{\lambda\in [0,1]}$ with $E_\lambda\subset \R^n$. and let us denote $\R^n \supset E := \bigcup_{\lambda\in [0, 1]} E_\lambda$.

Suppose that for some $\beta \in (0, n]$ and $\gamma\ge 1$, we have
\begin{itemize}
\item $\dim_{\mathcal{H}} E \le \beta$,
\item for any $\eps > 0$, and for any $x_\circ\in E_{\lambda_\circ}$ for some $\lambda_\circ\in [0, 1]$, there exists some $\rho = \rho(\eps, x_\circ, \lambda_\circ) > 0$ such that
\[
B_r(x_\circ) \cap E_\lambda = \varnothing \quad \textrm{for all}\quad r < \rho, \textrm{ and } \lambda > \lambda_\circ +r^{\gamma-\eps}.
\]
\end{itemize}
Then,
\begin{enumerate}
\item If $\beta <\gamma$, then $\dim_{\mathcal{H}} ( \{\lambda : E_\lambda\neq \varnothing\} ) \le \beta/\gamma < 1$.
\item If $\beta \ge\gamma$, then for $\mathcal{H}^1$-a.e. $\lambda\in \R $, we have $\dim_{\mathcal{H}}(E_\lambda)\le \beta - \gamma$.
\end{enumerate}
\end{lem}

We will also use the following lemma, analogous to the first part of Lemma~\ref{lem.CL} but dealing with the upper Minkowski dimension instead (which we denote $\overline{\dim}_{\mathcal{M}}$). We refer to \cite[Chapter 5]{Mat95} for more details on the upper/lower Minkowski content and dimension.

\begin{lem}
\label{lem.CL2}
Consider the family $\{E_\lambda\}_{\lambda\in [0,1]}$ with $E_\lambda\subset \R^n$. and let us denote $\R^n \supset E := \bigcup_{\lambda\in [0, 1]} E_\lambda$.

Suppose that for some $\beta \in [1, n]$ and $\gamma>\beta$, we have
\begin{itemize}
\item $\overline{\dim}_{\mathcal{M}} E \le \beta$,
\item for any $\eps > 0$, and for any $x_\circ\in E_{\lambda_\circ}$ for some $\lambda_\circ\in [0, 1]$, there exists some $\rho = \rho(\eps) > 0$ such that
\[
B_r(x_\circ) \cap E_\lambda = \varnothing \quad \textrm{for all}\quad r < \rho, \textrm{ and } \lambda > \lambda_\circ +r^{\gamma-\eps}.
\]
\end{itemize}
Then, $\overline{\dim}_{\mathcal{M}} ( \{\lambda : E_\lambda\neq \varnothing\} ) \le \beta/\gamma < 1$.
\end{lem}
\begin{proof}
Given $A \subset \R^n$, let us denote 
\begin{equation}
\label{eq.NAr}
N(A, r) := \min\left\{k : A\subset \cup_{i = 1}^k B_r(x_i) \quad\textrm{ for some $x_i \in \R^n$}\right\},
\end{equation}
the smallest number of $r$-balls needed to cover $A$. The upper Minkowski dimension of $A$ can then be defined as
\[
\overline{\dim}_{\mathcal{M}} A := \inf\bigg\{ s :  \limsup_{r\downarrow 0} N(A, r) r^s = 0\bigg\}
\]
(see \cite{Mat95}). Notice that the definition of upper Minkowski dimension does not change if we assume that the balls $B_r(x_i)$ from \eqref{eq.NAr} are centered at points in $A$ (by taking, for instance, balls with twice the radius). 

Since $\overline{\dim}_{\mathcal{M}} E \le \beta$, we have that for any $\delta > 0$, $N(E, r) = o(r^{\beta+\delta})$. Let us consider $N(E, r)$ balls of radius $r$ centered at $E$, $B_r(x_i)$, with $x_i\in E$. Thanks to our second hypothesis we have that 
\[
\bigcup_{\lambda\in [0, 1]}\{\lambda\}\times E_\lambda \subset \bigcup_{i = 1}^{N(E, r)} (\lambda(x_i) -r^{\gamma-\eps}, \lambda(x_i) +r^{\gamma-\eps})\times B_r(x_i),
\] 
where $x_i\in E_{\lambda(x_i)}$. Thus, 
\[
\{\lambda \in [0, 1]: E_\lambda \neq \varnothing \}\subset \bigcup_{i = 1}^{N(E, r)}(\lambda(x_i) -r^{\gamma-\eps}, \lambda(x_i) +r^{\gamma-\eps}),
\]
where the intervals are balls of radius $r^{\gamma-\eps}$. In particular, using that $N(E, r) = o(r^{\beta+\delta})$, we deduce that 
\[
\overline{\dim}_{\mathcal{M}}\left\{\lambda \in [0, 1]: E_\lambda \neq \varnothing \right\}\le \frac{\beta+\delta}{\gamma-\eps}.
\]
Since this works for any $\delta, \eps > 0$, we deduce the desired result. 
\end{proof}

\begin{rem}
Notice that Lemma~\ref{lem.CL} is somehow a generalization of the coarea formula. Namely, if we consider the case $\gamma = 1$, $\beta = n$, and $\eps = 0$, and we denote $E_\lambda$ the level sets of a Lipschitz function $f = f(\lambda)$ ($E_\lambda = f^{-1}(\lambda)$), the the coarea formula says that 
\[
\int_0^1 \mathcal{H}^{n-1}\left(f^{-1}(\lambda)\right)\, d\lambda = \int_{B_1} |\nabla f| < \infty,
\]
since $f$ is Lipschitz by assumption. In particular, $\mathcal{H}^{n-1}\left(f^{-1}(\lambda)\right) < \infty$ for $\mathcal{H}^1$-a.e. $\lambda\in [0, 1]$. This is used by Monneau in \cite{Mon03} for the classical obstacle problem.

This observation is also the reason why we do not expect to have a Minkowski analogous to Lemma~\ref{lem.CL} (2), as we did in Lemma~\ref{lem.CL2} for part (1).
\end{rem}

By applying the previous lemmas together with Proposition~\ref{prop.kap} we obtain the following result.

\begin{thm}
\label{thm.main000}
Let $u_\lambda$ solve \eqref{eq.thinobst_lam_vp}-\eqref{eq.uassump}. Let $\varphi\in C^{\tau, \alpha}$, and let $\kappa < \tau + \alpha$ and $\kappa \le \tau + \alpha -a$.

If $2+2s\le \kappa \le n+2s$, then,
\[
{\rm dim}_{\mathcal{H}}(\Gamma^\lambda_{\ge \kappa})\le n-\kappa+2s \quad\textrm{for a.e.}\quad \lambda\in [0, 1],
\]
On the other hand, if $\kappa > n+2s$, then
\[
\Gamma^\lambda_{\ge \kappa} = \varnothing \quad\textrm{for all}\quad \lambda\in [0, 1]\setminus \mathcal{E}_\kappa,
\]
where $\mathcal{E}_\kappa \subset [0, 1]$ is such that $\dim_{\mathcal{H}}(\mathcal{E}_\kappa) \le \frac{n}{\kappa-2s}$.

Furthermore, for any $h > 0$, if $\kappa > n+2s$, then 
\[
\Gamma^\lambda_{\ge \kappa}\cap B_{1-h} = \varnothing \quad\textrm{for all}\quad \lambda\in [0, 1]\setminus \mathcal{E}_{\kappa, h},
\]
where $\mathcal{E}_{\kappa, h} \subset [0, 1]$ is such that $\overline{\dim}_{\mathcal{M}}(\mathcal{E}_{\kappa, h}) \le \frac{n}{\kappa-2s}$.
\end{thm}
\begin{proof}
The proof of this result follows applying Lemmas~\ref{lem.CL} and \ref{lem.CL2} to the right sets. Indeed, we consider the sets
\[
E_\lambda := \Gamma_{\ge \kappa}^\lambda,\qquad E := \bigcup_{\lambda\in [0, 1]}E_\lambda.
\]
Notice that $E = \Gamma_{\ge \kappa}$, and we can take $\beta = n$ in Lemma~\ref{lem.CL}. On the other hand, we know that for any $\lambda_\circ \in [0, 1]$, $x_\circ \in E_{\lambda_\circ}$, there exists $\rho = \rho(x_\circ, \lambda_\circ) > 0$ such that
\[
B_r(x_\circ) \cap E_\lambda = \varnothing \quad \textrm{for all}\quad r < \rho, \textrm{ and } \lambda > \lambda_\circ +C_* r^{\kappa-2s}.
\]
thanks to Proposition~\ref{prop.kap}. That is, for any $\eps > 0$ there exists some $\rho = \rho(\eps, x_\circ, \lambda_\circ) > 0$ such that
\[
B_r(x_\circ) \cap E_\lambda = \varnothing \quad \textrm{for all}\quad r < \rho, \textrm{ and } \lambda > \lambda_\circ +r^{\kappa-2s-\eps}.
\]
and the hypotheses of Lemma~\ref{lem.CL} are fulfilled, with $\beta = n$ and $\gamma = \kappa-2s$. The result now follows by Lemma~\ref{lem.CL}.

The last part of the theorem follows by applying Lemma~\ref{lem.CL2} instead of Lemma~\ref{lem.CL}. We notice in this case that the dependence of $\rho$ on the point has been removed, but now it depends on $h> 0$. This forces the result to hold only in smaller balls $B_{1-h}$. 
\end{proof}

In particular, we can also deal with the set of free boundary points of infinite order.

\begin{cor}
\label{cor.main000}
Let $u_\lambda$ solve \eqref{eq.thinobst_lam_vp}-\eqref{eq.uassump}. Let $\varphi\in C^{\infty}$, and let $\Gamma_{\infty}^\lambda := \bigcap_{\kappa \ge 2} \Gamma_{\ge\kappa}^\lambda$. Then,
\[
\Gamma^\lambda_{\infty} = \varnothing \quad\textrm{for all}\quad \lambda\in [0, 1]\setminus \mathcal{E},
\]
where $\mathcal{E} \subset [0, 1]$ is such that $\dim_{\mathcal{H}}(\mathcal{E}) = 0$.

Furthermore, for any $h > 0$,
\[
\Gamma^\lambda_{\infty}\cap B_{1-h} = \varnothing \quad\textrm{for all}\quad \lambda\in [0, 1]\setminus \mathcal{E}_h,
\]
where $\mathcal{E}_h \subset [0, 1]$ is such that ${\dim}_{\mathcal{M}}(\mathcal{E}) = 0$.
\end{cor}
\begin{proof}
Apply Theorem~\ref{thm.main000} to $\Gamma_{\ge \kappa}^\lambda$ and let $\kappa \to \infty$.
\end{proof}

And we get that the free boundary points of order greater or equal than $2+2s$ are at most $(n-2)$-dimensional, for almost every $\lambda\in [0, 1]$.

\begin{cor}
\label{cor.22s}
Let $u_\lambda$ solve \eqref{eq.thinobst_lam_vp}-\eqref{eq.uassump}. Let $\varphi\in C^{4, \alpha}$. Then,
\[
{\rm dim}_{\mathcal{H}}(\Gamma^\lambda_{\ge 2+2s})\le n-2,
\]
for almost every $\lambda\in [0, 1]$.
\end{cor}
\begin{proof}
This is simply Theorem~\ref{thm.main000} with $\kappa = 2+2s$.
\end{proof}

On the other hand, combining the results from Sections~\ref{sec.Sch} and \ref{sec.dimGam2} with Lemma~\ref{lem.CL} we get the following regarding the free boundary points of order 2.

\begin{thm}
\label{thm.2grow}
Let $u_\lambda$ solve \eqref{eq.thinobst_lam_vp}-\eqref{eq.uassump}, and let $n \ge 2$. Then
\[
{\rm dim}_{\mathcal{H}}(\Gamma^\lambda_{2})\le n-2 \quad\textrm{for a.e.}\quad \lambda\in [0, 1].
\]
\end{thm}
\begin{proof}
The proof of this result follows applying Lemma~\ref{lem.CL} to the right sets. We consider
\[
E_\lambda := \Gamma_{2}^\lambda,\qquad E := \bigcup_{\lambda\in [0, 1]}E_\lambda = \Gamma_2.\]
Notice that $E$ has dimension $\mathcal{H}(E) = n-1$ by Proposition~\ref{prop.singdim}, so that we can take $\beta = n-1$ in Lemma~\ref{lem.CL}. On the other hand, we know that for any $\lambda_\circ \in [0, 1]$, $x_\circ \in E_{\lambda_\circ}$, and any $\eps > 0$, there exists $\rho = \rho(\eps, x_\circ, \lambda_\circ) > 0$ such that
\[
B_r(x_\circ) \cap E_\lambda = \varnothing \quad \textrm{for all}\quad r < \rho, \textrm{ and } \lambda > \lambda_\circ + r.
\]
thanks to Proposition~\ref{prop.2grow} (notice that $2\frac{2-s}{1+s}> 1$ for all $s\in (1/2, 1)$). That is, the hypotheses of Lemma~\ref{lem.CL} are fulfilled, with $\beta = n-1$ and $\gamma = 1$. The result now follows by Lemma~\ref{lem.CL}.
\end{proof}

In fact, the previous theorem is a particular case of the more general statement involving singular points given by the following proposition. We give it for completeness, although we do not need it in our analysis.

\begin{prop}
\label{prop.2mgrow}
Let $u_\lambda$ solve \eqref{eq.thinobst_lam_vp}-\eqref{eq.uassump}. Let $n \ge 2$ and let $\varphi\in C^{\tau, \alpha}$ for some $\tau \in \N_{\ge 4}$ and $\alpha\in (0, 1)$. Then, if $s \le \frac12$,
\[
{\rm dim}_{\mathcal{H}}(\Gamma^\lambda_{2})\le n-3 \quad\textrm{for a.e.}\quad \lambda\in [0, 1].
\]
Alternatively, if $s > \frac12$,
\[
{\rm dim}_{\mathcal{H}}(\Gamma^\lambda_{2})\le n-1-2\frac{2-s}{1+s} \quad\textrm{for a.e.}\quad \lambda\in [0, 1].
\]
Finally, if $m\in \N$ is such that $2m \le\tau$,
\[
{\rm dim}_{\mathcal{H}}(\Gamma^\lambda_{2m})\le n-1-2m+2s \quad\textrm{for a.e.}\quad \lambda\in [0, 1].
\]
\end{prop}
\begin{proof}
This proof simply follows by analysing the previous results more carefully. The first part follows exactly as Theorem~\ref{thm.2grow}, using Proposition~\ref{prop.2grow} and looking at each case separately.

Finally, regarding general singular points of order $2m$, the proof follows exactly as Theorem~\ref{thm.main000} using that $\Gamma_{2m}$ has dimension $n-1$ instead of $n$ thanks to Proposition~\ref{prop.singdim}.
\end{proof}

Finally, in order to control the size of points of homogeneity in the interval $(2, 2+2s)$, we refer to the following result by Focardi--Spadaro, that establishes that points in $\Gamma_*$ are lower dimensional with respect to the free boundary. The result in \cite{FS19} involves higher order points as well, but we state it in the explicit form it will be used below. 

\begin{prop}[\cite{FS19}]
\label{prop.smallother}
Let $u$ be a solution to the fractional obstacle problem with obstacle $\varphi\in C^{4, \alpha}$ for some $\alpha\in (0, 1)$,
\begin{equation}
\label{eq.top_b}
  \left\{ \begin{array}{rcll}
 L_a u&=&0 & \textrm{ in } B_1\setminus \left(\{x_{{n+1}} = 0\}\cap \{u = \varphi\}\right)\\
  L_a u&\le&0 & \textrm{ in } {B_1}\\
  u &\ge& \varphi & \textrm{ on } \{x_{{n+1}} = 0\}.
  \end{array}\right.
\end{equation}
Let $\theta\in (0, \alpha)$ and let us denote
\begin{equation}
\label{eq.G_star}
\tilde{\Gamma}_* := \bigcup_{\kappa\in(2, 2+2s)}  \bigg\{ x_\circ \in \de\Lambda(u) : \Phi_{\tau, \alpha, \theta} (0^+, \bar u^{x_\circ}) = n+1-2s+2\kappa\bigg\}.
\end{equation}
Then
\[
\dim_{\mathcal{H}}\tilde{\Gamma}_* \le n-2.
\]
Moreover, if $n = 2$, $\tilde{\Gamma}_* $ is discrete.
\end{prop}

Combining the previous results we obtain the following.

\begin{cor}
\label{cor.COR}
Let $u_\lambda$ solve \eqref{eq.thinobst_lam_vp}-\eqref{eq.uassump}. Let $\varphi\in C^{4, \alpha}$. Then,
\[
{\rm dim}_{\mathcal{H}}({\rm Deg}(u_\lambda))\le n-2,
\]
for almost every $\lambda\in [0, 1]$.
\end{cor}
\begin{proof}
This follows by combining the previous results. Notice that
\[
{\rm Deg}(u_\lambda) = \Gamma^\lambda \setminus  \Gamma^\lambda_{1+s} = \Gamma^\lambda_{2}\cup \tilde{\Gamma}_*(u_\lambda)\cup \Gamma_{\ge 2+2s}^\lambda.
\]
The result now follows thanks to Proposition~\ref{prop.smallother}, Corollary~\ref{cor.22s}, and Theorem~\ref{thm.2grow}.
\end{proof}
\begin{rem}
Following the proofs carefully, one can see that the previous result holds true for obstacles $\varphi\in C^{3, 1}$ if $s \le \frac12$. The condition $\varphi\in C^{4, \alpha}$ is only used whenever $s > \frac12$, since otherwise, in this case the previous methods do not imply the {\em smallness} of $\tilde{\Gamma}_*$.
\end{rem}

We can now prove the main results.

\begin{proof}[Proof of Theorem~\ref{thm.MAIN0}]
Notice that, by the Harnack inequality, there exists a constant $c$ such that $u_{\lambda + \eps} \ge g_\lambda + c\eps$ in $\de B_1\cap \{|x_{n+1}|\ge \frac12\}$. Thus, let us consider $w_\lambda = c^{-1} u_\lambda$, so that $w_\lambda$ fulfils \eqref{eq.uassump} and we can apply Corollary~\ref{cor.COR} to $w_\lambda$. Since $\Gamma_{\kappa}(w_\lambda) = \Gamma_{\kappa}(u_\lambda)$ for all $\kappa\in [3/2, \infty]$, $ \lambda\in [0,1]$,
\[
{\rm dim}_{\mathcal{H}}(\Gamma(u_\lambda)\setminus \Gamma_{3/2}(u_\lambda))\le n-2.
\]
We finish by recalling that $\Gamma_{3/2}(u_\lambda) = {\rm Reg}(u_\lambda)$ is open, and a $C^{\infty}$ $(n-1)$-dimensional manifold (see \cite{ACS08, KPS15, DS16}).
\end{proof}

\begin{proof}[Proof of Theorem~\ref{thm.MAIN01}]
With the same transformation as in the previous proof, the result now follows from Corollary~\ref{cor.main000}.
\end{proof}

\begin{proof}[Proof of Theorem~\ref{thm.MAIN1}]
Let us suppose that, after a rescaling if necessary, $\{\varphi > 0\}\subset B_1'\subset \R^n$.

We define $w_\lambda = v_\lambda + \lambda$, which fulfils a fractional obstacle problem, with obstacle $\varphi$, but with limiting value $\lambda$. Take the standard $a$-harmonic (i.e., with the operator $L_a$) extension of $w_\lambda$, which we denote $\tilde w_\lambda$, from $\R^n$ to $\R^{n+1}$.
 Thanks to \cite{CS07}, $\tilde w_\lambda$ fulfils a problem of the form \eqref{eq.thinobst_lam_vp} in $B_1\subset \R^{n+1}$.

 Moreover, by the Harnack inequality, $\tilde w_{\lambda+\eps}\ge \tilde w_\lambda + c\eps$ in $B_1\cap\{|x_{n+1}|\ge \frac12\}$ for some constant $c$. Now, the functions $c^{-1}\tilde w_\lambda$ fulfil \eqref{eq.uassump}, so that we can apply Corollary~\ref{cor.COR} to $c^{-1}\tilde w_\lambda$ to obtain
 \[
{\rm dim}_{\mathcal{H}}({\rm Deg}(v_\lambda)) = {\rm dim}_{\mathcal{H}}( \Gamma(v_\lambda)\setminus \Gamma_{1+s}(v_\lambda)) \le n-2.
\]
The result now follows since $\Gamma_{1+s}(v_\lambda) = {\rm Reg}(v_\lambda)$ is open, and a $C^{\infty}$ $(n-1)$-dimensional manifold (see \cite{ACS08, JN17, KRS19}).
\end{proof}

\begin{proof}[Proof of Theorem~\ref{thm.MAIN11}]
With the same transformation as in the previous proof, the result follows from Corollary~\ref{cor.main000}.
\end{proof}


\section{Examples of degenerate free boundary points}
\label{sec.examples}

Let us consider the thin obstacle problem in a domain $\Omega\subset \R^{{n+1}}$, with zero obstacle defined on $x_{{n+1}} = 0$. That is,
\begin{equation}
\label{eq.thinobst}
  \left\{ \begin{array}{rcll}
  -\Delta u&=&0 & \textrm{ in } \Omega\setminus \left(\{x_{{n+1}} = 0\}\cap \{u = 0\}\right)\\
  -\Delta u&\ge&0 & \textrm{ in } \Omega\\
  u &\ge& 0 & \textrm{ on } \{x_{{n+1}} = 0\} \\
  u & = & g & \textrm{ on } \de \Omega,
  \end{array}\right.
\end{equation}
for some continuous boundary values $g\in C^0(\de\Omega)$ such that $g > 0$ on $\de\Omega\cap \{x_{{n+1}}=0\}$.

\begin{proof}[Proof of Proposition~\ref{prop.singpoints}]
We will show that there exists some domain $\Omega$ and some boundary data $g$ such that the solution to \eqref{eq.thinobst} has a sequence of regular points (of order $3/2$) converging to a non-regular (singular) point (of order $2$). Then, the solution from Proposition~\ref{prop.singpoints} will be the solution here constructed restricted to any ball inside $\Omega$ containing such singular point, with its own boundary data (and appropriately rescaled, if necessary).

In order to build such a solution we will use \cite[Lemma 3.2]{BFR18}, which says that solutions to \begin{equation}
\label{eq.thinobst2}
  \left\{ \begin{array}{rcll}
  -\Delta u&=&0 & \textrm{ in } \Omega\setminus \left(\{x_{{n+1}} = 0\}\cap \{u = \varphi\}\right)\\
  -\Delta u&\ge&0& \textrm{ in } \Omega\\
  u &\ge& \varphi & \textrm{ on } \{x_{{n+1}} = 0\} \\
  u & = & 0 & \textrm{ on } \de \Omega,
  \end{array}\right.
\end{equation}
with $\Delta_{x'}\varphi \le -c_0 < 0$ and $\Omega$ convex and even in $x_{n+1}$ have a free boundary containing only regular points (frequency $3/2$) and singular points of frequency 2. In particular, they establish a non-degeneracy result stating that for any $x_\circ = (x_\circ', 0)\in \Gamma(u)$  then
\begin{equation}
\label{eq.nondeg}
\sup_{B'_r(x_\circ')}(u-\varphi) \ge c_1 r^2\quad\textrm{ for all } r \in (0, r_1),
\end{equation}
for some $r_1, c_1$ that do not depend on the point $x_\circ$. More precisely, they show it around points $x \in \{u > \varphi\}$ and then take the limit $x \to x_\circ \in \Gamma(u)$.

On the other hand, from their proof one can also show that in fact, the convexity on $\Omega$ can be weakened to convexity in $\Omega$ in the $\boldsymbol{e}_{n+1}$ direction.

Let us fix $n= 2$. Up to subtracting the right obstacle, we consider the problem
\begin{equation}
\label{eq.thinobst_phi}
  \left\{ \begin{array}{rcll}
  -\Delta u&=&0 & \textrm{ in } \Omega\setminus \left(\{x_{3} = 0\}\cap \{u = 0\}\right)\\
  -\Delta u&\ge&0 & \textrm{ in } \Omega\\
  u &\ge& \varphi_t & \textrm{ on } \{x_{3} = 0\} \\
  u & = & 0 & \textrm{ on } \de \Omega,
  \end{array}\right.
\end{equation}
for some analytic obstacle $\varphi_t$, and some domain $\Omega$ smooth, convex and even in $x_3$, to be chosen.

Let $\varphi_t(x) = t-(1-x_1^2)^2 -4x_2^2$. Notice that, in the thin space, $\Delta_{x'} \varphi_t = -12x_1^2-4\le -4 $, so that, by the result in \cite{BFR18}, under the appropriate domain $\Omega$, the points on the free boundary $\Gamma(u_t)$ are either regular (with frequency 3/2) or singular (with frequency 2), and we have non-degeneracy \eqref{eq.nondeg}. Let $\Omega' := \{x'\in \R^2 : (1-x_1^2)^2+4x_2^2\le 2\}$, and take any bounded, convex in $x_3$, and even in $x_3$ extension of $\Omega'$, $\Omega$. Then, if $t = 2$ and $\Omega\subset \{|x_3|\le 1\}$, the solution $u_2$ to \eqref{eq.thinobst_phi} is exactly equal to the solution to\footnote{To see this, we compare $u_2$ with the harmonic extension of $\varphi_2$, $\tilde \varphi_2(x_1, x_2, x_3) = \varphi_2(x_1, x_2) + 2 x_3^2 + 6x_1^2x_3^2 -x_3^4$.}
\[
  \left\{ \begin{array}{rcll}
  \Delta u_2&=&0 & \textrm{ in } \Omega\setminus \{x_{3} = 0\}\\
  u_2 & = & 0 & \textrm{ on } \de \Omega\\
    u_2 & = & \varphi_2 & \textrm{ on } \{x_3 = 0\},
  \end{array}\right.
\]
so that, in particular, the contact set is full.

Notice that, when $t < 0$, the contact set is empty, $\Lambda(u_t) = \varnothing$, and when $t = 0$ the contact set is two points, $p_\pm = (\pm 1, 0, 0)$ (which, in particular, are singular points). Notice, also, that the contact set is always closed and is monotone in $t$, in the sense that $\Lambda(u_{t_1}) \subseteq \Lambda(u_{t_2})$ if $t_1 \le t_2$. Let us say that a set is $p_\pm$-connected if the points $p_+$ and $p_-$ belong to the same connected component. Then, there exists some $t^*\in (0, 2]$ such that $\Lambda(u_t)$ is not $p_\pm$-connected for $t < t^*$, and is $p_\pm$-connected for $t > t^*$. Notice, also, that since $\Lambda(u_t)\subset \{x' : \varphi_t \ge 0\}$ then $t^* > 1$.


We claim that $\Lambda(u_{t^*})$ is $p_\pm$-connected and has a set of regular points converging to a singular point.

Let us first show that $\Lambda(u_{t^*})$ is $p_\pm$-connected. Suppose it is not. That is, $\Lambda(u_{t^*})$ is a closed set with $p_\pm$ on different connected components. On the other hand, $\Lambda(u_{t})$ is compact and $p_\pm$-connected for $t > t^*$, and nested ($\Lambda(u_{t})\subset \Lambda(u_{t'})$ for $t < t'$). Take
\[
\tilde{\Lambda}_{t^*} := \bigcap_{t \in (t^*, 2]}\Lambda(u_t),
\]
then $\tilde{\Lambda}_{t^*}$ is $p_\pm$-connected (being the intersection of compact $p_\pm$-connected nested sets), and $\Lambda(u_{t^*}) \subsetneq\tilde{\Lambda}_{t^*}$, since $\Lambda(u_{t^*})$ is not $p_\pm$-connected. In particular, there exists some $x_\circ\in \Lambda(u_{t})$ for all $t > t^*$ such that $x_\circ\not\in\Lambda(u_{t^*})$. But, by continuity, this is not possible: $0 < (u_{t^*}-\varphi_{t^*})(x_\circ) = \lim_{t\downarrow t^*}(u_{t}-\varphi_{t})(x_\circ) = 0$. Therefore, $\Lambda(u_{t^*})$ is $p_{\pm}$-connected.

Take $\Lambda^p(u_{t^*})$ to be the connected component containing both $p_+$ and $p_-$. Then, $\de\Lambda^p(u_{t^*})$ must contain at least one singular point. Indeed, suppose it is not true. In this case, all points in $\de\Lambda^p(u_{t^*})$ are regular, and in particular, $\Lambda^p(u_{t^*})$ is a compact connected set with smooth boundary, with all points of the boundary having positive density (in $\{x_3 = 0\}$), and therefore $\left(\Lambda^p(u_{t^*})\right)^\circ$ is also connected. Let us denote $\Lambda^{p}_\pm(u_{t})$ the corresponding connected components of $\Lambda(u_t)$ containing $p_\pm$ for $t < t^*$ (notice that, by definition of $t^*$, $\Lambda^{p}_+(u_{t})\neq \Lambda^{p}_-(u_{t})$. Then,
\[
\Lambda_{t<t^*}^{p, \circ} := \left(\bigcup_{t < t^*} \left(\Lambda_+^p(u_t)\right)^\circ\right)\cup \left(\bigcup_{t < t^*} \left(\Lambda_-^p(u_t)\right)^\circ\right) \subsetneq \left(\Lambda^p(u_{t^*})\right)^\circ,
\]
given that the left-hand side is not connected, and the right-hand side is. Take $y_\circ\in \left(\Lambda^p(u_{t^*})\right)^\circ\setminus \Lambda_{t<t^*}^{p, \circ}$, so that around $y_\circ$ the non-degeneracy \eqref{eq.nondeg} holds for any $t < t^*$. Then, there exists some $r_\circ > 0 $, $r_1 > r_\circ$ (where $r_1$ is defined in \eqref{eq.nondeg}) such that $B_{r_\circ}'(y_\circ)\subset \Lambda^p(u_{t^*})$, so that $u_{t^*}-\varphi_{t^*}|_{B_{r_\circ}'(y_\circ)} \equiv 0$ and
\[
0 < c_1 r_\circ^2 \le \lim_{t \uparrow t^*} \sup_{B_r'(x_\circ')} (u_t-\varphi_t) = \sup_{B_r'(x_\circ')} (u_{t^*}-\varphi_{t^*}) = 0,
\]
a contradiction. That is, not all points on $\de\Lambda^p(u_{t^*})$ are regular. By \cite{BFR18}, then there exist some degenerate (singular) point of frequency 2, $x_D\in \de\Lambda^p(u_{t^*})$. Now consider $\Gamma_D$, the connected component in $\de\Lambda^p(u_{t^*})$ containing $x_D$. Since the density of the contact set around singular points is zero, if $\Gamma_D$ consist exclusively of singular points, then $\Gamma_D$ itself is the whole connected component $\Lambda^p(u_t)$, and $p_\pm \in \Gamma_D$ are singular points. Nonetheless, for small $t > 0$, $\Lambda(u_t)$ contains a neighbourhood of $p_\pm$, which contradicts the singularity of $p_\pm$. Therefore, $\Gamma_D$ is not formed exclusively of singular points, and then there exists a sequence of regular points converging to a singular point.
\end{proof}

Now, before proving Proposition~\ref{prop.secondexample}, let us show the following lemma.

\begin{lem}
\label{lem.exta}
Let $m\in \N_{> 0}$, and let $\eta\in C^\infty_c(B_2)$ such that $\eta \equiv 1$ in $B_1$. Let $u_+ = \max\{u, 0\}$ and $u_- = -\min\{u, 0\}$. Then,
\[
(-\Delta)^s \left[(x_1)_+^{2m+1+s} \eta\right] - C_{m, s} (x_1)_-^{2m+1-s} \in C^\infty(B_{1/2}),
\]
for some positive  constant $C_{m, s}>0$ depending only on $n$, $m$, and $s$.
\end{lem}
\begin{proof}
We consider the extension problem from $\R^n$ to $\R^{n+1}$. Namely, let us denote $u_1$ the extension of $(x_1)_+^{2m+1+s} \eta$, that is, $u_1$ solves
\[
  \left\{ \begin{array}{rcll}
  L_a u_1&= &0& \textrm{ in } \R^{n+1}\cap \{x_{n+1} > 0\}\\
  u_1(x', 0) &=&(x_1)_+^{2m+1+s} \eta & \textrm{ for } x'\in \R^{n}\\
  u_1(x) &\to &0 & \textrm{ as } |x|\to \infty,
  \end{array}\right.
\]
where $a = 1-2s$. Then, we know that
\[\left\{(-\Delta)^s \left[(x_1)_+^{2m+1+s} \eta\right]\right\}(x') = \lim_{y\downarrow 0} y^a\partial_{x_{n+1}} u_1(x', y)
\]
for $x'\in \R^n$. On the other hand, let $u_2$ be the unique $a$-harmonic extension of $(x_1)_+^{2m+1+s}$ from $\R^n$ to $\R^{n+1}$. That is, $u_2$ is homogeneous (of degree $2m+1+s$), and fulfils
\[
  \left\{ \begin{array}{rcll}
  L_a u_2&= &0& \textrm{ in } \R^{n+1}\cap \{x_{n+1} > 0\}\\
  u_2(x', 0) &=&(x_1)_+^{2m+1+s} & \textrm{ for } x'\in \R^{n}.
  \end{array}\right.
\]
The fact that such solution exists, and that $\lim_{y\downarrow 0} y^a\partial_{x_{n+1}} u_2(x', y) = 0$ if $x_1 > 0$, follows, for example, from \cite[Proposition A.1]{FS18}. On the other hand, notice that, since $u_2$ is $(2m+1+s)$-homogeneous, we have that, $\lim_{y\downarrow 0} y^a\partial_{x_{n+1}} u_2(x', y) = C_{m, s} |x_1|^{2m+1-s}$ for $x_1 < 0$, so that, in all,
\[
\lim_{y\downarrow 0} y^a\partial_{x_{n+1}} u_2(x', y) = C_{m, s} (x_1)_-^{2m+1-s}.
\]
Again, by \cite[Proposition A.1]{FS18} $u_2$ is a solution to the thin obstacle problem with operator $L_a$, so $C_{m, s}> 0$ (otherwise, it would not be a supersolution for $L_a$).

Let now $v = u_1 - u_2$. Notice that $v$ fulfils
\[
  \left\{ \begin{array}{rcll}
  L_a v&= &0& \textrm{ in } \R^{n+1}\cap \{x_{n+1} > 0\}\\
  v(x', 0) &=&(x_1)_+^{2m+1+s}(\eta - 1) & \textrm{ for } x'\in \R^{n}.
  \end{array}\right.
\]
In particular, $v(x', 0) = 0$ in $B_1'$. Let us denote $D^\alpha_{x'} v$ a derivative in the $x'\in \R^n$ direction of $v$, with multi-index $\alpha = (\alpha_1, \alpha_2, \dots, \alpha_n, 0)$. Then $D^\alpha_{x'}v$ is such that
\[
  \left\{ \begin{array}{rcll}
  L_a D^\alpha_{x'}v&= &0& \textrm{ in } B_1\cap \{x_{n+1} > 0\}\\
  D^\alpha_{x'}v (x', 0) &=&0 & \textrm{ for } x'\in B_1'.
  \end{array}\right.
\]
Then, by estimates for the operator $L_a$, we know that, if we define
\[
w_\alpha(x') := \lim_{y\downarrow 0} y^a\partial_{x_{n+1}} D^\alpha v(x', y),\qquad w_0(x') := \lim_{y\downarrow 0} y^a\partial_{x_{n+1}} v(x', y),
\]
then $w_\alpha$ satisfies $w_\alpha\in C^{\beta}(B_{1/2})$ for some $\beta > 0$ (see \cite[Proposition 4.3]{CSS08} or \cite[Proposition 2.3]{JN17}). In particular, since $w_\alpha = D^\alpha w_0$, we have that $w_0\in C^{|\alpha|+\beta}(B_{1/2})$. Since this works for all multi-index $\alpha$, $w_0\in C^\infty(B_{1/2})$.

Thus, combining the previous steps,
\begin{align*}
(-\Delta)^s \left[(x_1)_+^{2m+1+s} \eta\right] - C_{m, s} (x_1)_-^{2m+1-s}&  =  \lim_{y\downarrow 0} y^a\partial_{x_{n+1}} (u_1(x', y) - u_2(x', y))\\
& = \lim_{y\downarrow 0} y^a\partial_{x_{n+1}} v(x', y)\\
& = w_0 \in C^\infty(B_{1/2}),
\end{align*}
as we wanted to see.
\end{proof}

We are now in disposition to give the proof of Proposition~\ref{prop.secondexample}.
\begin{proof}[Proof of Proposition~\ref{prop.secondexample}]
We divide the proof into two steps. In the first step, we show the results holds up to an intermediate claim, that will be proved in the second step.
\\[0.1cm]
{\bf Step 1.} Thanks to \cite[Theorem 4]{Gru15} or \cite[Section 2]{AR19}, we have that $(-\Delta)^s(d^s\eta)\in C^\infty(\overline{\Omega^c})$ for any $\eta\in C^\infty$ with sufficient decay at infinity. Here, $d$ denotes any $C^\infty$ function (with at most polynomial growth at infinity) such that in a neighbourhood of $\Omega$ coincides with the distance to $\Omega$, and $d|_\Omega \equiv 0$.

In particular, once $d$ is fixed, we know that for any $k\in \N$,
\[
(-\Delta)^{s}(d^{k+s}) = f\in C^\infty(\overline{\Omega^c}),
\]
and, if we make sure that $d>0$ in $\Omega^c$, with exponential decay at infinity, we get
\[
|f(x)|\le \frac{C}{1+|x|^{n+2s}}.
\]
Define, for some $g$ with the previous decay, $|g(x)|\le C(1+|x|^{n+2s})^{-1}$, $\varphi_g$ such that
\[
(-\Delta)^s \varphi_g = g,
\]
that is, one can take
\[
\varphi_g(x) = I_{2s} g(x) := c \int_{\R^n}\frac{g(y)}{|x-y|^{n-2s}} dy.
\]

Notice that
\[
\begin{split}
|\varphi_g(x)|& \le C\int_{\R^n}\frac{dy}{(1+|y|^{n+2s})|x-y|^{n-2s}}   \\
& \le C\int_{|y-x|\ge \frac{|x|}{2}}\frac{dy}{(1+|y|^{n+2s})|x-y|^{n-2s}} +C\int_{|y-x|\le \frac{|x|}{2}}\frac{dy}{(1+|y|^{n+2s})|x-y|^{n-2s}} \\
& \le \frac{C}{|x|^{n-2s}}\int_{|y-x|\ge \frac{|x|}{2}}\frac{dy}{1+|y|^{n+2s}}+ \frac{C}{1+|x|^{n+2s}}\int_{|y-x|\le \frac{|x|}{2}}\frac{dy}{|x-y|^{n-2s}},
\end{split}
\]
where we are using that if $|y-x|\le \frac{|x|}{2}$ then $|y|\ge \frac{|x|}{2}$ by triangular inequality. Notice also that
\[
\int_{|y-x|\le \frac{|x|}{2}}\frac{dy}{|x-y|^{n-2s}}= \int_{B_{|x|/2}} \frac{dz}{|z|^{n-2s}} = \int_{0}^{|x|/2} r^{2s-1} dr = C |x|^{2s}.
\]
In all, also using that $\varphi(x)$ is bounded around the origin, we obtain that
\[
|\varphi_g(x)|\le \frac{C}{1+|x|^{n-2s}}.
\]

Now let us define $v = d^{k+s}$. We claim that, if $k = 2m+1$ for some $m\in \N_{> 0}$, then $v$ fulfils
\begin{equation}
\label{eq.claimfreq}
  \left\{ \begin{array}{rcll}
  (-\Delta)^{s} v&\ge&\bar f & \textrm{ in } \R^n\\
    (-\Delta)^{s} v&=&\bar f & \textrm{ in } \{v > 0\}\\
    v & \ge & 0& \textrm{ in } \R^n,
  \end{array}\right.
\end{equation}
where $\bar f$ is some appropriate $C^\infty$ extension of $f$ inside $\Omega$. Then, if we define
\[
u:= v + \varphi_{-\bar f},
\]
$u$ fulfils,
\[
  \left\{ \begin{array}{rcll}
  (-\Delta)^{s} u&\ge&0 & \textrm{ in } \R^n\\
    (-\Delta)^{s} u&=&0 & \textrm{ in } \{u > \varphi_{-\bar f}\}\\
    u & \ge & \varphi_{-\bar f}& \textrm{ in } \R^n,
  \end{array}\right.
\]
and notice that, since $v> 0$ in $\Omega^c$ and $v = 0$ in $\Omega$, by definition, we have that the contact set is exactly equal to $\Omega$. Moreover, by the growth of $v$ at the boundary, the free boundary points are of frequency $k+s$. Also, by the decay at infinity of $v$ and $\varphi_{-\bar f}$, $u\to 0$ at infinity.
\\[0.1cm]
{\bf Step 2.} We still have to show that, for an appropriate choice of $\bar f$, \eqref{eq.claimfreq} holds for $k = 2m+1$. Notice that, in fact, in $\Omega^c$ we know that $f$ is $C^\infty$. Moreover, we only have to show the claim for a neighbourhood of $\de\Omega$ inside $\Omega$, given that exactly at the boundary we expect a \emph{unique} extension of $f$ (that is, all derivatives are prescribed at the boundary).

That is, if we let $\Omega_\delta := \{x\in \Omega : \textrm{dist}(x, \de\Omega) < \delta\}$, we have to show that there exists some $\delta > 0$ small enough such that $(-\Delta)^s v \ge \bar f$ in $\Omega_\delta$, where we recall that $\bar f$ is a $C^\infty$ extension of $f\in C^\infty(\overline{\Omega^c})$ inside $\Omega$.

Let $z_\circ\in \de \Omega$. After a translation and a rotation, we assume that $z_\circ = 0$ and $\nu(0, \de \Omega) = \boldsymbol{e}_1$, where $\nu(0, \de\Omega)$ denotes the outward normal to $\de\Omega$ at $0$. After rescaling if necessary, let us assume that we are working in $B_1$, that each point in $B_1$ has a unique projection onto $\de \Omega$, and that $d|_{B_1\cap \Omega^c} = {\rm dist}(\cdot, \Omega)$. Moreover, again after a rescaling if necessary (since $\Omega$ is a $C^\infty$ domain), let us assume that
\begin{equation}
\label{eq.assump}
\{y_1\le -|(y_2,\dots,y_n)|^2\}\cap B_1 \subset \Omega\cap B_1  \subset \{y_1\le |(y_2,\dots,y_n)|^2\}\cap B_1,
\end{equation}
so that, in particular, $\{-t\boldsymbol{e}_1 : t\in (0, 1)\}\subset \Omega$.

Let $\eta\in C^\infty_c(B_2)$ such that $\eta\equiv 1$ in $B_1$, and let $u_+ = \max\{u, 0\}$ denote the positive part, and $u_- = -\min\{u, 0\}$ the negative part. Let $\alpha  =2m+1+s$, and define
\[
u_1(x) := (x_1)_+^\alpha\eta,\qquad w(x) := v(x) - u_1(x) = d^\alpha(x) - (x_1)_+^\alpha\eta.
\]

Notice that, by Lemma~\ref{lem.exta},
\begin{equation}
\label{eq.cinfdiff}
(-\Delta)^s u_1(x) - C_{m, s} (x_1)_-^{2m+1-s} \in C^\infty(B_{1/2}),
\end{equation}
for some positive constant $C_{m, s}>  0$.

We begin by claiming that
\begin{equation}
\label{eq.claim}
w_1(x_1) := [(-\Delta)^s w] (x_1,0,\dots,0)\in C^{2m+1-s+\eps}((-1/2, 1/2)),
\end{equation}
for some $\eps > 0$.

Indeed, let any $z_1\in (-1/2, 1/2)$.
Let us denote for $\gamma\in (0, 1]$,  $\delta_{\boldsymbol{e}_1, h}^{(\gamma)}$ the incremental quotient in the $\boldsymbol{e}_1$ direction of length $0<h<1/4$ and order $\gamma$; that is,
\[
\delta_{\boldsymbol{e}_1, h}^{(\gamma)} F (y_\circ) := \frac{|F(y_\circ+h\boldsymbol{e}_1)-F(y_\circ)|}{|h |^\gamma}.
\]
Since $d \equiv (x_1)_+$ on $\{x_2 =\dots= x_n = 0\}\cap B_1$, we have that $w(x_1, 0,\dots,0) = 0$ on $(-1, 1)$. Now notice that, for any $\ell\in \N$, $\gamma\in (0, 1]$,
\begin{equation}
\label{eq.w1}
\delta_{\boldsymbol{e}_1, h}^{(\gamma)} \frac{d^\ell}{dx_1^\ell} w_1 (z_1) = \left\{\delta_{\boldsymbol{e}_1, h}^{(\gamma)}\de_{\boldsymbol{e}_1}^\ell [(-\Delta)^s w] \right\}(z_1,0,\dots,0) = \int_{\R^n} \frac{\delta_{\boldsymbol{e}_1, h}^{(\gamma)}\de_{\boldsymbol{e}_1}^\ell w (\bar z_1 + y)}{|y|^{n+2s}} \,dy,
\end{equation}
where $\bar z_1 = \{z_1,0,\dots,0\}\in \R^n$, and we are using that $\delta_{\boldsymbol{e}_1, h}^{(\gamma)} \de_{\boldsymbol{e}_1}^\ell w (\bar z_1) = 0$. In order to show \eqref{eq.claim},  we will bound
\begin{equation}
\label{eq.claimbound}
\lim_{h\downarrow 0} \left|\delta_{\boldsymbol{e}_1, h}^{(\gamma)} \frac{d^\ell}{dx_1^\ell} w_1 (z_1)\right|\le C\quad\textrm{in} \quad B_{1/2},
\end{equation}
for some $C$, for $\ell = 2m$ and for $\gamma = 1-s+\eps$ for some $\eps  >0$.


We need to separate into different cases according to $\bar z_1+y$. Notice that the the integral in \eqref{eq.w1} is immediately bounded in $\R^n\setminus B_{1/2}$ because $w \in C^\alpha$ and the integrand is thus bounded by $C|y|^{-n-2s}$. We can, therefore, assume that $y \in B_{1/2}$ so that $\bar z_1+y\in B_1$.

Let us start by noticing that, from \eqref{eq.cinfdiff}, together with the fact that $(-\Delta)^s v$ is smooth in $\Omega^c$, we already know that $w_1\in C^\infty([0, 1/2))$, so that we only care about the case $z_1 < 0$.

Let $z_1 < 0$, so that $\bar z_1\in \Omega$. If $\bar z_1 + y \in \Omega \cap\{x_1 < 0\} \cap B_1$, then $w(\bar z_1 + y) = 0$. If $\bar z_1 + y \in \Omega \cap\{x_1 > 0\} \cap B_1$, then $|w(\bar z_1 + y)| = |z_1+y_1|^{\alpha}$ and $|\de^\ell_{\boldsymbol{e}_1} w|(\bar z_1+y) = C |z_1+y_1|^{\alpha-\ell}\le C |y|^{2(\alpha-\ell)}$; where we are using that $z_1+y_1 \le |(y_2,\dots,y_n)|^2\le |y|^2$, see \eqref{eq.assump}. Similarly, $\lim_{h\downarrow 0}|\delta_{\boldsymbol{e}_1, h}^{(\gamma)} \de^\ell_{\boldsymbol{e}_1} w|(\bar z_1+y) \le C |z_1+y_1|^{\alpha-\ell-\gamma}\le C |y|^{2(\alpha-\ell-\gamma)}$.

Conversely, if  $\bar z_1 + y \in \Omega^c \cap\{x_1 < 0\} \cap B_1$, $|w(\bar z_1+y)| = d^\alpha(\bar z_1+y)$ and $|\de_{\boldsymbol{e}_1}^\ell w|(\bar z_1+y)\le Cd^{\alpha-\ell}(\bar z_1+y)\le C|y|^{2(\alpha-\ell)}$, where we are using \eqref{eq.assump} again. Taking the incremental quotients, $\lim_{h\downarrow 0}|\delta_{\boldsymbol{e}_1, h}^{(\gamma)}\de_{\boldsymbol{e}_1}^\ell w|(\bar z_1+y)\le Cd^{\alpha-\ell-\gamma}(\bar z_1+y)\le C|y|^{2(\alpha-\ell-\gamma)}$

Finally, if $\bar z_1+y\in \Omega^c\cap \{x_1> 0\}\cap B_1$, both terms in the expression of $w$ are relevant. Using that $|a^\beta-b^\beta|\le C |a-b||a^{\beta-1}+b^{\beta-1}|$ we obtain that
\[
|w(\bar z_1+y)|\le C|d - u_1|\left(d^{\alpha-1}+u_1^{\alpha-1}\right) (\bar z_1+y).
\]

Notice that on $\{x_2 = \dots = x_n =0\}\cap B_1$, $d = u_1$ and $\de_i d = \de_i u = 0$ for $2\le i \le n$, so that in fact $|d-u_1|(\bar z_1 + y)\le C |y|^2$. On the other hand, we also have that $d^{\alpha-1}(\bar z_1+ y)\le C|y|^{\alpha-1}$, so that
\begin{equation}
\label{eq.estimates}
|w(\bar z_1+y)|\le C|y|^{\alpha+1}.
\end{equation}

Notice, also, that $w\in C^\alpha$ (i.e., $\nabla^{\ell+1} w \in C^{s}$). By classical interpolation inequalities for H\"older spaces (or fractional Sobolev spaces with $p = \infty$) we know that, if $0<\gamma < 1$,
\[
\|\nabla^\ell w\|_{C^\gamma(B_{r}(\bar z_1))} \le C \|\nabla^{\ell+1} w\|_{C^{s}(B_{r}(\bar z_1))}^{\frac{\ell+\gamma}{\alpha}} \|w\|_{L^\infty(B_{r}(\bar z_1))}^{\frac{1+s-\gamma}{\alpha}}
\]
(see, for instance, \cite[Theorem 6.4.5]{BL76}). Thus, in our case we have that
\begin{equation}
\label{eq.wrong}
\lim_{h\downarrow 0}\left|\delta_{\boldsymbol{e}_1, h}^{(\gamma)}\frac{d^\ell}{dx_1^\ell} w\right|(\bar z_1 + y) \le C |y|^{(\alpha+1)\frac{1+s-\gamma}{\alpha}}.
\end{equation}


Thus, putting all together we obtain that
\[
\lim_{h\downarrow 0 }\left|\delta_{\boldsymbol{e}_1, h}^{(\gamma)}\de_{\boldsymbol{e}_1}^\ell w\right|(\bar z_1+y)\le C\max\left\{|y|^{2(\alpha-\ell-\gamma)}, |y|^{(\alpha+1)\frac{1+s-\gamma}{\alpha}}\right\}.
\]
If we want \eqref{eq.claimbound} to hold, we need (by checking \eqref{eq.w1})
\begin{equation}
\label{eq.weneed}
2(\alpha-\ell-\gamma)> 2s\qquad \textrm{ and }\qquad (\alpha+1)\frac{1+s-\gamma}{\alpha} > 2s,
\end{equation}
for some $1-s<\gamma < 1$, and $\ell = 2m$ (recall we need to show $\gamma = 1-s+\eps$ for some $\eps  > 0 $). The first inequality holds as long as $\gamma < 1$. The second inequality will hold if
\[
\gamma < 1+s-\frac{2s\alpha}{\alpha+1} = 1-\frac{\alpha-1}{\alpha+1} s.
\]
Thus, we can choose $\gamma = 1-s+\eps$ with $0<\eps <\frac{2}{\alpha+1}s$ and \eqref{eq.claim} holds with this $\eps$.

Now, combining \eqref{eq.claim}-\eqref{eq.cinfdiff} we obtain that
\[
f_v:= [(-\Delta)^s v ](x_1,0,\dots,0) - C_{m, s}(x_1)_-^{2m+1-s} \in C^{2m+1-s+\eps}((-1/2, 1/2)).
\]
In particular, if we recall that $\bar f\in C^\infty(B_1)$ is a $C^\infty$ extension of $(-\Delta)^s v$ inside $\Omega$, and noticing that $f_v-\bar f(x_1,0,\dots,0) \equiv 0$ for $x_1 > 0$, we have that $\bar f(\cdot, 0,\dots, 0)- f_v\in C^{2m+1-s+\eps}((-1/2, 1/2))$ and
\[
f_v - \bar f(x_1, 0, \dots, 0)  = o(|x_1|^{2m+1-s+\eps}),
\]
or
\[
[(-\Delta)^s v ](x_1,0,\dots,0) = C_{m, s}(x_1)_-^{2m+1-s} + \bar f (x_1, 0,\dots,0) +o(|x_1|^{2m+1-s+\eps}).
\]
Thus, since $C_{m, s} > 0$, $[(-\Delta)^s v ](x_1,0,\dots,0) \ge \bar f(x_1, 0, \dots, 0)$ if $|x_1|$ is small enough (depending only on $n$, $m$, $s$, and $\Omega$), as we wanted to see.

We have that, for a fixed $\bar f$ extension of $f$ inside $\Omega$, $(-\Delta)^s v \ge \bar f$ in $\Omega_\delta$ for some small $\delta > 0$ depending only on $n$, $m$, $s$, and $\Omega$. Up to redefining $\bar f$ in $\Omega\setminus\Omega_{\delta/2}$, we can easily build an $\bar f\in C^\infty$ such that $(-\Delta)^s v \ge \bar f$ in $\Omega$, as we wanted to see.
\end{proof}

To finish, we study the points of order infinity. To do that, we start with the following proposition.
\begin{prop}
\label{prop.Ccontact}
Let $\mathcal{C}\subset B_1\subset \R^n$ be any closed set. Then, there exists a non-trivial solution $u$ and an obstacle $\varphi\in C^\infty(\R^n)$ such that
\[
  \left\{ \begin{array}{rcll}
  (-\Delta)^{s} u&\ge&0 & \textrm{ in } \R^n\\
    (-\Delta)^{s} u&=&0 & \textrm{ in } \{u > \varphi\}\\
    u & \ge & \varphi & \textrm{ in } \R^n,
  \end{array}\right.
\]
and $\Lambda(u)\cap B_1 = \{u = \varphi\}\cap B_1 = \mathcal{C}$.
\end{prop}
\begin{proof}
Take any obstacle $\psi\in C^\infty(\R^n)$ such that ${\rm supp}\,\psi\subset\subset B_1(2\boldsymbol{e}_1)$, with $\psi> 0$ somewhere, and take the non-trivial solution to
\[
  \left\{ \begin{array}{rcll}
  (-\Delta)^{s} u&\ge&0 & \textrm{ in } \R^n\\
    (-\Delta)^{s} u&=&0 & \textrm{ in } \{u > \psi\}\\
    u & \ge & \psi & \textrm{ in } \R^n.
  \end{array}\right.
\]
Notice that $ u > \psi$ in $B_1$ (in particular, $u\in C^\infty(B_1)$). Let $f_{\mathcal{C}}$ be any $C^\infty$ function such that $0\le f_{\mathcal{C}} \le 1$ and $\mathcal{C} = \{f_{\mathcal{C}} = 0\}$.

Now let $\eta\in C^\infty_c(B_{3/2})$ such that $\eta \ge 0$ and $\eta \equiv 1 $ in $B_1$. Consider, as new obstacle, $\varphi = \psi + \eta(u-\psi)(1-f_\mathcal{C}) \in C^\infty(B_1)$. Notice that $u - \varphi\ge 0$. Notice, also, that for $x\in B_1$, $(u - \varphi)(x) =  0$ if and only if $x\in \mathcal{C}$. Thus, $u$ with obstacle $\varphi$ gives the desired result.
\end{proof}

And now we can provide the proof of Proposition~\ref{prop.inf_points}:
\begin{proof}[Proof of Proposition~\ref{prop.inf_points}]
The proof is now immediate thanks to Proposition~\ref{prop.Ccontact}, since we can choose as contact set any closed set with boundary of dimension greater or equal than $n - \eps$ for any $\eps > 0$, and points of finite order are at most $(n-1)$-dimensional. 
\end{proof}

\section{The parabolic Signorini problem}
\label{sec.parab}
We consider now the parabolic version of the thin obstacle problem. Given $(x_\circ, t_\circ)\in \R^{n+1}\times\R$, we will use the notation
\[
\begin{split}
Q_r (x_\circ, t_\circ)& := B_r(x_\circ)\times(t_\circ-r^2,t_\circ]\subset \R^{n+1}\times \R,\\
Q_r'(x_\circ', t_\circ) & := B_r'(x_\circ')\times(t_\circ-r^2, t_\circ]\subset \R^n \times \R,\\
Q_r^+((x_\circ', 0), t_\circ) & := B_r^+((x_\circ', 0))\times(t_\circ-r^2, t_\circ]\subset \R^{n+1}\times\R.
\end{split}
\]

We will denote, $Q_r = Q_r(0, 0)$, $Q_r' = Q_r'(0, 0)$ and $Q_r^+ = Q_r^+(0,  0)$. We consider the problem posed in $Q_1^+ := B_1^+\times(-1, 0]$ for some fixed obstacle
\[
\varphi :B_1'\to \R,\quad\varphi\in C^{\tau, \alpha}(\overline{B_1'}),\quad\tau \in \N_{\ge 2},\alpha\in(0,1],
\]
that is,
\begin{equation}
\label{eq.parab_thin}
\left\{
\begin{array}{rcll}
\de_t u - \Delta u & = & 0,&\quad\textrm{in}\quad Q_1^,\\
\min\{u-\varphi,\de_{x_{n+1}}u\} & = & 0, & \quad\textrm{on}\quad Q_1'.
\end{array}
\right.
\end{equation}

The free boundary for \eqref{eq.parab_thin} is given by
\[
\Gamma(u) := \de_{Q_1'}\{(x', t) \in Q_1' : u(x', 0, t) > \varphi(x')\},
\]
where $\de_{Q_1'}$ denotes the boundary in the relative topology of $Q_1'$. For this problem, it is more convenient to study the {\em extended} free boundary, defined by
\[
\overline{\Gamma}(u) := \de_{Q_1'}\{(x', t) \in Q_1' : u(x', 0, t) = \varphi(x'), \,\de_{x_{n+1}} u (x', 0, t)  =0\},
\]
so that $\overline{\Gamma}(u) \supset \Gamma(u)$. This distinction, however, will not come into play in this work.

In order to study \eqref{eq.parab_thin}, one also needs to add some boundary condition on $(\de B_1\times(-1, 0])\cap \{x_{n+1} > 0\}$. Instead of doing that, we will assume the additional hypothesis $u_t > 0$ on $(\de B_1\times(-1, 0])\cap \{x_{n+1} > 0\}$. That is, there is actually some time evolution, and it makes the solution grow. Recall that such hypothesis is (somewhat) necessary, and natural in some applications (see subsection~\ref{ss.parab}).

Notice, also, that if $u_t > 0$ on the spatial boundary, by strong maximum principle applied to the caloric function $u_t$ in $Q_1\cap \{x_{n+1} > \frac12\}$, we know that $u_t > c > 0$ for $x_{n+1}> \frac12$. Thus, after dividing $u$ by a constant, we may assume $c = 1$, and thus, our problem reads as
\begin{equation}
\label{eq.parab_thin2}
\left\{
\begin{array}{rcll}
u_t - \Delta u & = & 0&\quad\textrm{in}\quad Q_1^+\times(-1, 0],\\
\min\{u-\varphi,\de_{x_{n+1}}u\} & = & 0 & \quad\textrm{on}\quad Q_1',\\
u_t & > & 0& \quad\textrm{on}\quad (\de B_1\times(-1, 0])\cap \{x_{n+1} > 0\},\\
u_t & \ge & 1& \quad\textrm{in}\quad Q_1\cap\{x_{n+1}>\frac12\}.
\end{array}
\right.
\end{equation}

In order to deal with the order of free boundary points, one requires the introduction of heavy notation, analogous to what has been presented in the elliptic case, but for the parabolic version. We will avoid that by focusing on the main property we require about the order of the extended free boundary points:
\begin{defi}

Let $(x_\circ, t_\circ)\in \overline{\Gamma}(u)\cap Q_{1-h}$ be an extended free boundary point. We define
\[
\overline{u}^{x_\circ, t_\circ}(x, t)  := u((x+x_\circ',x_{n+1}), t+t_\circ) -  \varphi(x'+x_\circ') + Q^{x_\circ}_\tau(x') - Q_\tau^{x_\circ, 0}(x', x_{n+1}),
\]
where $Q_\tau^{x_\circ}$ is the Taylor polynomial of order $\tau$ of $\varphi$ at $x_\circ$, and $Q_\tau^{x_\circ, 0}$ is its harmonic extension to $\R^{n+1}$.

We say that $(x_\circ, t_\circ)\in \overline{\Gamma}(u)\cap Q_{1-h}$ is an extended free boundary point of order $\ge \kappa$, $(x_\circ, t_\circ)\in \Gamma_{\ge \kappa}$, where $2 \le \kappa \le \tau$, if
\[
|\overline{u}^{x_\circ, t_\circ}|\le Cr^{\kappa}\quad\textrm{in}\quad Q_r^+,
\]
for all $r < \frac{h}{2}$, and for some constant $C$ depending only on the solution $u$.
\end{defi}

Notice that, in particular, the points of order greater or equal than $\kappa$ as defined in \cite{DGPT17} fulfil the previous definition. Notice, also, that we have denoted by $\Gamma_{\ge \kappa}$ the set of points of order $\ge\kappa$.

Thus, we can proceed to prove the following proposition, analogous to Proposition~\ref{prop.kap}:
\begin{prop}
\label{prop.boundtime}
Let $h > 0$ small, and let $(x_\circ, t_\circ)\in Q_{1-h}^+\cap \Gamma_{\ge \kappa}$ with $t_\circ < -h^2$, where $2 \le \kappa \le 3$. Then,
\[
u(\cdot, t_\circ + C_* t^{\kappa-1}) > \varphi\quad\textrm{in}\quad B_t'(x_\circ'),\quad\textrm{for all}\quad 0<t < T_h,
\]
for some constant $C_*$ depending only on $n$, $h$, $u$, and $T_h$ depending only on $n$, $h$, $\tau$, $\kappa$, $u$.
\end{prop}
\begin{proof}
Let us assume, for simplicity in the notation, that $x_\circ = 0$, and $t_\circ = -\frac12$, and we denote $\overline{u} := \overline{u}^{0, -1/2}$. Notice that, by the parabolic Hopf Lemma, since $\overline{u}_t \ge 0$ in $Q_1$ and $\overline{u}_t \ge 1$ in $Q_1\cap \{x_{n+1}\ge \frac12\}$ we have that for some constant $c$ and for any $\sigma > 0$,
\[
\overline{u}_t \ge c\sigma\quad\textrm{in}\quad (B^+_{1/2}\cap\{x_{n+1}\ge \sigma\})\times[-1/2, 0].
\]

Notice, also, that since $(0, -1/2)\in \R^{n+1}\times\R$ is an extended free boundary point of order $\ge \kappa$, we have that, for $r > 0$ small enough,
\begin{equation}
\label{eq.extkap}
\overline{u}(\cdot, -1/2+s)\ge \overline{u}(\cdot, -1/2)\ge -Cr^{\kappa}\quad\textrm{in}\quad B_r^+,
\end{equation}
for $s\ge 0$ by the monotonicity of the solution in time.

On the other hand, since $\overline{u}_t \ge cr\sigma$ in $\{x_{n+1} \ge r\sigma\}$, we have that
\[
\overline{u}(\cdot, -1/2+ s) \ge c(r\sigma) s + \overline{u}(\cdot, -1/2)\quad\textrm{in}\quad \{x_{n+1}\ge r\sigma\} \quad\textrm{for}\quad s \ge 0.
\]
As in \eqref{eq.extkap}, this gives
\[
\overline{u}(\cdot, -1/2+ s) \ge c(r\sigma) s -Cr^\kappa\quad\textrm{in}\quad \{x_{n+1}\ge r\sigma\}\cap B_r^+ \quad\textrm{for}\quad s \ge 0.
\]
Let $w(y, \zeta) = \overline{u} (ry, -1/2+r^2\zeta)$. Then we have that
\[
w(y, \zeta) \ge -Cr^\kappa, \quad\textrm{for}\quad y\in B_1^+\quad\textrm{for}\quad \zeta\ge 0,
\]
and
\[
w(y, \zeta) \ge c(r\sigma)r^2 \zeta -Cr^\kappa,\quad\textrm{for}\quad y\in \{y_{n+1}\ge \sigma\}\cap B_1^+\quad\textrm{for}\quad \zeta \ge 0.
\]
Notice, also, that since
\[
|(\de_t - \Delta) \overline{u}| = o(r^{\tau -2})\quad\textrm{in}\quad B_r^+,
\]
then
\[
|(\de_\zeta - \Delta_y) w|  = o(r^\tau) \quad\textrm{in}\quad B_1^+.
\]

Considering now $\bar w (y, \zeta):= \frac{\sigma}{Cr^\kappa}w(y, \zeta)$, we have that
\[
\bar w(y, \zeta) \ge -\sigma, \quad\textrm{for}\quad y\in B_1^+~\textrm{  and  }~\zeta\ge 0,
\]
\[
\bar w(y, \zeta) \ge cr^{3-\kappa} \sigma^2 \zeta -\sigma,\quad\textrm{for}\quad y\in \{y_{n+1}\ge \sigma\}\cap B_1^+~\textrm{  and  }~\zeta \ge 0,
\]
and
\[
|(\de_\zeta - \Delta_y) \bar w| \le \sigma \quad\textrm{in}\quad B_1^+,
\]
for $r > 0$ small enough. Let us take $\zeta = C_*r^{\kappa-3}$, for some $C_*$ depending on $n$ and $\sigma$ such that  $cr^{3-\kappa} \sigma^2 \zeta -\sigma \ge 1$. Then, by \cite[Lemma 11.5]{DGPT17} (which is the parabolic version of Lemma~\ref{lem.epslem} for $a = 0$), there exists some $\sigma_\circ > 0$ depending on $n$ such that if $\sigma \le \sigma_\circ$, then $\bar{w}(\cdot, C_*r^{\kappa-3})> 0$ in $\overline{B_{1/2}^+}$. In particular, recalling the definition of $\bar w$, this yields the desired result.
\end{proof}

As in the elliptic case, the non-regular part of the free boundary is $\Gamma_{\ge 2}$ (see \cite[Proposition 10.8]{DGPT17}). Thanks to Proposition~\ref{prop.boundtime} we will obtain a bound on the dimension of  $\Gamma_{\ge \kappa}\cap\{t = t_\circ\}$ for almost every time $t_\circ\in (-1, 0]$ if $\kappa > 2$. For the limiting case, $\kappa = 2$, one has to proceed differently, analogous to what has been done in the elliptic case.

Let us start by defining the set $\Gamma_2$. We say that a point $(x_\circ, t_\circ) \in \overline{\Gamma}(u)\cap Q_{1-h}^+$ belongs to $\Gamma_2$, $(x_\circ, t_\circ) \in \Gamma_2\cap Q_{1-h}^+$, if parabolic blow-ups around that point converge uniformly to a parabolic 2-homogeneous polynomial.

Namely, consider a fixed test function $\psi\in C^\infty_c(\R^n)$ such that ${\rm supp}\, \psi\subset B_{h}$, $0\le \psi \le 1$, $\psi \equiv 1$ in $B_{h/2}$, and $\psi(x', x_{n+1}) = \psi(x', -x_{n+1})$. Then $u^{x_\circ, t_\circ}(x, t) \psi(x) $ can be considered to be defined in $\R^n_+\times (-h^2, 0]$, and we denote
\[
H_u^{x_\circ, t_\circ}(r) :=  \frac{1}{r^2}\int_{-r^2}^0\int_{\R^n_+} \bar u^{x_\circ, t_\circ}(x, t) \psi(x) G(x, t) \, dx\, dt,
\]
where $G(x, t)$ is the backward heat kernel in $\R^{n+1}\times\R$,
\[
G(x, t) = \left\{
\begin{array}{ll}
(-4\pi t)^{-\frac{n+1}{2}} e^{\frac{|x|^2}{4t}} & \textrm{if }t < 0,\\
0 & \textrm{if }t \ge 0.
\end{array}
\right.
\]
We then define the rescalings
\[
u_r^{x_\circ, t_\circ}(x, t) := \frac{\bar u^{x_\circ, t_\circ}(rx, r^2 t)}{H_u^{x_\circ, t_\circ}(r)^{1/2}}.
\]
Then, we say that $(x_\circ, t_\circ) \in \Gamma_2$ if for every $r_j \downarrow 0$, there exists some subsequence $r_{j_k}\downarrow 0$ such that
\[
u_{r_{j_k}}^{x_\circ, t_\circ}\to p_2^{x_\circ, t_\circ}\quad\textrm{uniformly in compact sets},
\]
for some parabolic 2-homogeneous caloric polynomial $p_2^{x_\circ, t_\circ} = p_2^{x_\circ, t_\circ}(x, t)$ (i.e., $p_2(\lambda x, \lambda^2 t) = \lambda^2 p_2(x, t)$ for $\lambda > 0$), which is a global solution to the parabolic Signorini problem. The exis\-tence of such polynomial, the uniqueness of the limit, and its properties, are shown in \cite[Proposition 12.2, Lemma 12.3, Theorem 12.6]{DGPT17}. Moreover, by the classification of free boundary points performed in \cite{DGPT17} we know that
\[
\Gamma(u) = {\rm Reg}(u) \cup \Gamma_{\ge 2}.
\]

In addition, by \cite[Proposition 4.5]{Shi18} there are no free boundary points with frequency belonging to the interval $(2, 2+ \alpha_\circ )$ for some $\alpha_\circ> 0$ depending only on $n$. Thus,
\begin{equation}
\label{eq.fbparab}
\Gamma(u) = {\rm Reg}(u) \cup \Gamma_{2}\cup \Gamma_{\ge 2+\alpha_\circ}.
\end{equation}

\begin{prop}
\label{prop.parab1}
The set $\Gamma_2$ defined as above is such that
\[
\dim_{\mathcal{H}}(\Gamma_2\cap \{t = t_\circ\}) \le n-2,\quad\textrm{for a.e.}\quad t_\circ\in (-1, 0].
\]
\begin{proof}
We separate the proof into two steps.
\\[0.1cm]
{\bf Step 1.}
By \cite[Theorem 12.6]{DGPT17}, we know that
\[
\bar u^{x_\circ, t_\circ}(x, t) = p_2^{x_\circ, t_\circ}(x, t) + o(\|(x, t)\|^2),
\]
where $\|(x, t)\| = (|x|^2+|t|)^{1/2}$ is the parabolic norm. Here $p_2^{x_\circ, t_\circ}$ is a polynomial, parabolic 2-homogeneous global solution to the parabolic Signorini problem. In particular, it is at most linear in time.  On the other, since $u_t \ge 0$ everywhere, the same occurs with the parabolic blow-up up, i.e., $p_2^{x_\circ, t_\circ}$ is non-decreasing in time. All this implies that $p_2^{x_\circ, t_\circ}$ is actually constant in time, so that we have that $p_2^{x_\circ, t_\circ} = p_2^{x_\circ, t_\circ}(x)$ is an harmonic, second-order polynomial in $x$, non-negative on the thin space $\{x_{n+1} = 0\}$, and we have
\[
\bar u^{x_\circ, t_\circ}(x, t) = p_2^{x_\circ, t_\circ}(x) + o(\|(x, t)\|^2).
\]

On the other hand, also from \cite[Theorem 12.6]{DGPT17}, $\Gamma_2 \ni (x_\circ, t_\circ)\mapsto p_2^{x_\circ, t_\circ}$ is continuous. These last two conditions correspond to Proposition~\ref{prop.uniconv} and Proposition~\ref{prop.whitapp} from the elliptic case. In particular, one can apply Whitney's extension theorem as in Proposition~\ref{prop.singdim} to obtain that the set
\[
\pi_x \Gamma_2 := \{ x\in \R^{n+1} : (x, t) \in \Gamma_2\textrm{ for some $t \in (-1, 0]$} \},
\]
is contained in the countable union of $(n-1)$-dimensional $C^1$ manifolds. That is,
\[
\dim_{\mathcal{H}}(\pi_x \Gamma_2) \le n-1,
\]
$\pi_x \Gamma_2$ is $(n-1)$-dimensional.
\\[0.1cm]
{\bf Step 2.} Thanks to Step 1, and by Proposition~\ref{prop.boundtime} with $\kappa = 2$, proceeding analogously to Theorem~\ref{thm.main000} by means of Lemma~\ref{lem.CL}, we reach the desired result.
\end{proof}
\end{prop}

\begin{prop}
\label{prop.parab2}
Let $a > 0$. Then,
\[
\dim_{\mathcal{H}}(\Gamma_{\ge 2+a}\cap \{t = t_\circ\}) \le n-1-a,\quad\textrm{for a.e.}\quad t_\circ\in (-1, 0],
\]
\end{prop}
\begin{proof}
The result follows by Proposition~\ref{prop.boundtime} with $\kappa = 2+a$, proceeding analogously to Theorem~\ref{thm.main000} by means of Lemma~\ref{lem.CL}.
\end{proof}

We can now give the proof of the main result regarding the parabolic Signorini problem.

\begin{proof}[Proof of Theorem~\ref{thm.parab_MAIN}]
Is a direct consequence of \eqref{eq.fbparab}, Proposition~\ref{prop.parab1}, and Proposition~\ref{prop.parab2} with $a = \alpha_\circ$ depending only on $n$, given by \cite[Proposition 4.5]{Shi18}. The regularity of the free boundary follows from \cite[Theorem 11.6]{DGPT17}.
\end{proof}

\end{document}